\documentclass[11pt]{article}

\usepackage[english]{babel}
\usepackage[utf8]{inputenc}
\usepackage[T1]{fontenc}
\usepackage{lmodern}
\usepackage{amssymb,amsmath,amsfonts}
\usepackage{xspace}
\usepackage{graphicx}
\usepackage{hyperref}
\usepackage{url}
\usepackage{stmaryrd}
\usepackage[babel=true]{csquotes}
\usepackage{amsthm}
\usepackage{subfigure}
\usepackage{mathrsfs}
\usepackage{float}

\usepackage{tikz}

\theoremstyle{plain}
\newtheorem{thm}{Theorem}[section]
\newtheorem{lm}[thm]{Lemma}
\newtheorem{prop}[thm]{Proposition}

\theoremstyle{definition}
\newtheorem{dfn}[thm]{Definition}
\newtheorem{rmk}[thm]{Remark}

\def\Hil{\mathcal{H}\xspace}

\def\N{\mathbb{N}\xspace}
\def\Z{\mathbb{Z}\xspace}

\def\R{\mathbb{R}\xspace}
\def\C{\mathbb{C}\xspace}

\def\T{\mathbb{T}\xspace}

\hypersetup{pdfborder={0000}, colorlinks=true, linkcolor=blue, citecolor=red}

\newcommand{\ga}{\gamma} 
\newcommand{\Ga}{\Gamma}
\newcommand{\om}{\omega} 
\newcommand{\al}{\alpha} 
\newcommand{\si}{\sigma}
\newcommand{\la}{\lambda} 
\newcommand{\Om}{\Omega}

\newcommand{\be}{\beta} 
\newcommand{\Ci}{\mathcal{C}^{\infty}} 
\newcommand{\con}{\overline}
\newcommand{\bigo}{\mathcal{O}}

\newcommand{\Hilb}{\mathcal{H}}

\newcommand{\wt}{\widetilde}
\newcommand{\op}{\operatorname}

\newcommand{\s}{\op{sub}}

\newcommand\blfootnote[1]{%
  \begingroup
  \renewcommand\thefootnote{}\footnote{#1}%
  \addtocounter{footnote}{-1}%
  \endgroup
}

\begin{document}

\title{Quantum propagation for Berezin-Toeplitz operators}

\author{Laurent Charles, Yohann Le Floch}

\maketitle

\begin{abstract}
We describe the asymptotic behaviour of the quantum propagator generated by a
Berezin-Toeplitz operator with real-valued principal symbol. We also give
precise asymptotics for smoothed spectral projectors associated with the
operator in the autonomous case; this leads us to introducing quantum states
associated with immersed Lagrangian submanifolds. These descriptions involve
geometric quantities of two origins, coming from lifts of the Hamiltonian flow
to the prequantum bundle and the canonical bundle respectively. The latter are
the main contribution of this article and are connected to the Maslov indices
appearing in trace formulas, as will be explained in a forthcoming paper.
\end{abstract}


\section{Introduction}

In quantum mechanics, the evolution of a state $\Psi_t$ under the influence of a Hamiltonian $\hat{H}$ can be described using Schr\"odinger's equation
\[ i \hbar \frac{d}{dt} \Psi_t = \hat{H} \Psi_t. \]
Under suitable assumptions on $\hat{H}$, the solutions to this equation are of the form $\Psi_t = U_{\hbar,t} \Psi_0$ where $U_{\hbar,t}$ is an operator called the quantum propagator. This propagator is the quantum analogue of the Hamiltonian flow in classical Hamiltonian mechanics. This analogy can be studied rigorously by investigating the so-called semiclassical limit $\hbar \to 0$ in which, if $\hat{H}$ quantizes the classical Hamiltonian $H$, $U_{\hbar,t}$ is expected to behave like the Hamiltonian flow of $H$. This statement has been given a precise meaning by studying the Schwartz kernel of $U_{\hbar,t}$ in different regimes of $\hbar$ and $t$, for semiclassical Schr\"odinger Hamiltonians $\hat{H} = -\hbar^2 \Delta + V$ on $T^* \R^d$, and more generally for $\hbar$-pseudodifferential operators on $T^* \R^d$ or $T^* X$ with $X$ a compact Riemannian manifold; see Section \ref{subsect:comparison} for a longer discussion and references. 

Here we are interested in a different setting where the underlying phase space is a compact symplectic manifold; then the quantum states are sections of a power of some well-chosen line bundle, and this power is the relevant semiclassical parameter. This setting naturally appears in several problems from physics, such as the study of spin systems in the large spin limit, coherent states, and the quantum Hall effect, \emph{cf.} for example \cite{Lieb,Glau,KleMaMaWi}. The limit of large power of a suitable line bundle is also very important in complex geometry, see for instance \cite{Don, RubZel, GuSt, Biq}.
 
The aim of our work is to understand, in this context, the geometric invariants appearing in the asymptotic description of the quantum propagator (and its counterparts, smoothed spectral projectors) in the semiclassical limit. As can be seen from other results in the same direction \cite{BPU,ZeZh,Ioos}, this is in fact non trivial and different authors have different, more or less explicit, ways to compute these invariants. Here we obtain expressions that are both completely natural and easily computable. This will be particularly important in forthcoming papers in which we revisit trace formulae: the explicit asymptotics that we obtain here will allow us to derive those in a direct way and with a precise control of the quantities they involve, in particular the Maslov-like indices contained in the subprincipal contributions.

\subsection{Berezin-Toeplitz operators}
\label{subsect:BTO}

Let $M^n$ be a compact complex manifold endowed with two Hermitian
holomorphic line
bundles $L$ and $L'$. We assume that $L$ is positive, meaning that the curvature of
its Chern connection is $\frac{1}{i} \om$ with $\om \in \Om ^2 ( M , \R) \cap \Om^{(1,1)}(M) $
positive.  For any
positive integer $k$, let $\mathcal{H}_k$ be the space of holomorphic sections
of $L^k \otimes L'$. The scalar product of sections of $L^k \otimes
L'$ is defined as the integral of the pointwise scalar product against the Liouville
volume form $\mu = \frac{\om^n}{n!}$.

Given a function $f \in \Ci (M)$, the
Berezin-Toeplitz operator $T_k(f)$ is the endomorphism of $\Hilb_k$ such that
$$ \langle T_k (f) u, v  \rangle = \langle f u , v \rangle, \qquad \forall u,v
\in\Hilb_k .$$
We are interested in the semiclassical limit $k\rightarrow +\infty$ and the
techniques we use allow to consider more general families $T := (T_k ( f(\cdot,
k)))$ where the multiplicator itself depends on $k$ and has an
expansion of the form $f(\cdot, k) = f_0 + k^{-1} f_1 + \ldots$ with coefficients $f_{\ell}
\in \Ci ( M)$. We will also consider time-dependent sequences $f(\cdot,
t,k)$ with an expansion with time-dependent coefficients.

We call the family $T$ a {\em Toeplitz operator}, $f_0$ its {\em principal symbol}  and $f_1 +
\frac{1}{2} \Delta f_0$ its {\em subprincipal symbol}. Here $\Delta$ is the
holomorphic Laplacian associated with the K\"ahler form $\om$, so $ \Delta = \sum h^{ij}
\partial_{z_i} \partial_{\con{z}_j}$ when $\om = i \sum h_{ij} dz_i \wedge d\con{z}_j$.
The reason why we introduce this subprincipal symbol is merely that it
simplifies the subleading calculus. 

Typically, if $T$ and $S$ are two Toeplitz
operators with principal symbols $f$ and $g$, then $TS$ and $ ik
[T,S]$ are Toeplitz operators with respective principal symbols $fg$ and the Poisson
bracket $\{ f, g\}$ with respect to $\om$ \cite{BoGu, BMS}. If now $T$ and $S$ have identically zero
subprincipal symbols, then the subprincipal symbols of $TS$ and $ik [T,S]$ are
$\frac{1}{2i} \{ f,g \}$ and $- \om_1 ( X,Y )$ respectively \cite{C2007}, where
$X$, $Y$ are the Hamiltonian vector fields of $f$ and $g$ and $\om_1$ is the
real two-form given by  $\om_1= i ( \Theta_{L'} - \frac{1}{2} \Theta_{K })$, with $\Theta_{L'}$, $\Theta_{K}$ the
curvatures of the Chern connections of $L'$ and of the canonical bundle $K$.

\subsection{The quantum propagator}

It is a well-known result that the solution of the Schr\"odinger equation for a
pseudo-differential operator is a Fourier integral operator associated with the
Hamiltonian flow of its principal symbol. 
Our first result is the Toeplitz analogue of this fact. Consider a time-dependent Toeplitz operator
$(T_{k,t})$ with principal symbol $(H_t)$  and subprincipal symbol $(H^{\s}_{t})$.
The quantum propagator generated by $T_{k,t}$ is the smooth path ($U_{k,t}$, $t
\in \R$)  of (unitary in case $T_{k,t}$ is self-adjoint)
maps of  $\Hilb_k$ satisfying the Schr\"odinger equation
\begin{gather} \label{eq:schrod_equation}
(ik)^{-1} \frac{d}{dt} U_{k,t} + T_{k,t} U_{k,t} = 0, \qquad U_{k,0} = \op{id}
.
\end{gather}
Our goal is to describe the Schwartz kernel of $U_{k,t}$, which by
definition is 
$$ U_{k,t} (x,y) = \sum_{i =1 }^{ d_k}  (U_{k,t} \psi_i) (x) \otimes \overline{\psi}_{i}
( y) \in (L^k \otimes L')_x \otimes (\con{L}^k \otimes \con{L}')_y$$ 
where $d_k = \dim \Hilb_k$ and $(\psi_i)$ is any orthonormal basis of $\Hilb_k$. In the sequel, we will view
$U_{k,t} (x,y)$ as a map from $ (L^k \otimes L')_y$ to $(L^k \otimes L')_x$
using the scalar product of $(L^k \otimes L')_y$.

As we will see, when the principal symbol $(H_t)$ is real, this Schwartz kernel is concentrated on the graph of the
Hamiltonian flow $\phi_t$ of $H_t$. Here the symplectic form $\om$ is $i$
times the curvature of $L$, and 
the Hamiltonian vector field $X_t$ is such that
\begin{gather} 
\om ( X_t, \cdot ) + dH_t =
0.
\end{gather}
To describe the asymptotic behavior of $U_{k,t} ( \phi_t (x) , x)$, we
need to introduce two lifts of $\phi_t$, the first one  to $L$ and the second
one to the canonical
bundle $K$ of $M$. The relevant structures on $L$ will be its
metric and its connection, which is generally called the prequantum structure. 

\subsubsection*{Parallel transport and prequantum lift} 
If $A \rightarrow M$ is a Hermitian line bundle endowed with a connection $\nabla$, the
{\em parallel transport} in $A$ along a path $\ga: [0, \tau ] \rightarrow M$ is a
unitary map
$$\mathcal{T}( A , \ga) : A_{\ga (0)} \rightarrow A_{\ga ( \tau)}$$
which can be computed as follows: if $u$ is a frame of $\ga^*A$, then
$\mathcal{T} ( A , \ga) u ( 0 ) = \exp (
i \int_\ga \al ) u ( \tau)$, where $\al \in \Om^1 ( [0, \tau], \R)$ is the connection one-form defined
in terms of the covariant derivative of $u$ by  $\nabla u = -i \al \otimes u$.
In particular we can lift by parallel transport the Hamiltonian flow $\phi_t$. We set $\mathcal{T}_t^A (x) := \mathcal{T} ( A,
\phi_{[0, t]} (x) ) : A_x \rightarrow A_{\phi_t (x)}$. 

The {\em prequantum lift} of the
Hamiltonian flow $\phi_t$ to $L$ is defined by  
\begin{gather} \label{eq:prequantum_lift} 
\phi_{t}^L  (x) = e^{  \frac{1}{i}  \int_0^t H_r ( \phi_r ( x) ) \; dr}
\mathcal{T}^L_t (x). 
\end{gather}
This lift has  an interest independently of Toeplitz operators: by the
Kostant-Souriau theory, $\phi^L_t$ is the unique (up to a
phase) lift of $\phi_t$  which preserves the metric and the connection of $L$. Furthermore, if $x$
belongs to a contractible periodic trajectory with period $T$, so that we can
define the action $\mathcal{A}(x,T) \in \R$, then $\phi_T^L(x): L_x \to L_x$ is the multiplication by $\exp
( i \mathcal{A}(x,T))$. 

In our results, it is actually the $k$-th power $(\phi^L_t)^{\otimes k}$ that will appear, with some corrections
involving the subprincipal data $L'$ and $H^{\s}_{t}$,
more precisely we will see  
\begin{gather} \label{eq:preq_lift_with_subprincipal_correction}
e^{ \frac{1}{i}
 \int_0^{t} H_{r}^{\s} ( \phi_r(x)) \; dr} \Bigr[ \phi_{t}^L (x) \Bigl]^{\otimes k} \otimes \mathcal{T}_{t}^{L'}
(x)  : (L^k \otimes L')_x
\rightarrow (L^k \otimes L')_{\phi_t (x)} .
\end{gather}
So we merely replace  $L$ by $L^k \otimes L'$ and
$H_t$ by $k H_t + H^{\s}_{t}$ in \eqref{eq:prequantum_lift}.

\subsubsection*{Holomorphic determinant and lift to the canonical bundle} 

The second ingredient we need is an invariant of the complex and symplectic
structures together. If $g : S \rightarrow S'$ is a linear symplectomorphism
between two $2n$-dimensional symplectic
vector spaces both endowed with linear complex structures, we define an
isomorphism $K(g) : K(S) \rightarrow K(S')$ between the canonical lines $K(S)
=  \wedge^{n,0} S^* $ , $K(S') = \wedge ^{n,0} (S')^*$ characterized by  
\begin{gather} \label{eq:defkg}
  K(g) (\al )( g_* u) = \al (u), \qquad
\forall \alpha \in K(S), u \in \wedge^n S. 
\end{gather}
Equivalently, if $E$ and $E'$ are the $(1,0)$-spaces of $S$ and
$S'$ respectively, we have decompositions
\begin{gather*} 
S \otimes \C = E \oplus \con{E}, \qquad S' \otimes \C = E' \oplus \con{E}',
\qquad g\otimes \op{id}_{\C} = \begin{pmatrix} g^{1,0} & * \\ * &* 
\end{pmatrix}
\end{gather*}
with $g^{1,0} :E \rightarrow E'$. Then $K(g) $ is the dual map of
$\det g^{1,0} : \wedge^n E \rightarrow \wedge ^nE'$ in the sense that $ K(g)(\al) (
(\det g^{1,0}) u) = \al (u)$ for any $\al \in K(S)$ and $u \in \wedge^{n} E$.

This holomorphic determinant has a nice structure 
in terms of the polar decomposition of linear symplectic maps. When $S=
S'=\R^{2n}$ with its usual complex structure $j$, $g = g_1 g_2$ where $g_1$ and
$g_2$ are both symplectic, $g_1$ commutes with $j$, and $g_2$ is symmetric positive
definite. Then $K(g)$ is a complex number whose inverse is 
\begin{gather} \label{eq:pol_formula}
{\det}_{\C} \, g^{1,0} = \left( \prod_{i=1}^n \frac{\la_i + \la_i^{-1}}{2} \right) \,
{\det}_{\C} \, g_1
\end{gather}
with $0< \la_1 \leqslant \ldots \leqslant \la_n < 1 < \la_n^{-1} \leqslant
\ldots \leqslant \la_1^{-1}$ the eigenvalues of $g_2$, and $\op{det}_{\C} g_1$
 the determinant of $g_1$ viewed as a $\C$-linear endomorphism of $\C^n$. Indeed, ${\det}_{\C} \, g^{1,0} = {\det}_{\C} \, g_2^{1,0} \,
{\det}_{\C} \, g_1$ and one readily computes ${\det}_{\C} \, g_2^{1,0}$ using the diagonalization of $g_2$ in an orthonormal basis $(e_1, \ldots, e_n, j e_1, \ldots, j e_n)$ with $g_2 e_{\ell} = \lambda_{\ell} e_{\ell}$ and $g_2 j e_{\ell} = \lambda_{\ell}^{-1} j e_{\ell}$. 
This formula generalizes for $(S,j) \neq (S', j')$, with now linear symplectic maps $g_1 : S \rightarrow S'$ and $g_2 : S
\rightarrow S$, where $g_1 \circ j = j'\circ g_1$, and $g_2$ is positive definite for the
Euclidean structure $\om ( \cdot, j\cdot )$ of $S$. The complex determinant of
$g_1$ may be viewed as a map from $\op{det}_{\C} (S,j)$ to $\op{det}_{\C}
(S',j')$ or equivalently from $\wedge ^{\op{top}} E$ to $ \wedge
^{\op{top}} E'$. 

This definition provides us with a lift $\mathcal{D}_t$ of the Hamiltonian flow
$\phi_t$ to the canonical bundle $K = \wedge ^{(n,0)} T^*M$, defined
by
\begin{gather} \label{eq:defDt}
\mathcal{D}_t(x) = K(T_x \phi_t) : K_x \rightarrow K_{\phi_t(x)}. 
\end{gather}
We have another lift of $\phi_t$ to the
canonical bundle which is the parallel transport $\mathcal{T}_t^K$. Define
the complex number $\rho_t(x)$ such that $ \mathcal{D}_t(x) = \rho_t(x) \mathcal{T}_t^K(x)$. 

\subsubsection*{The result}

\begin{thm} \label{th:intro1}
  Let $(U_{k,t})$ be the quantum propagator of a time-dependent Toeplitz
  operator $(T_{k,t})$ with real principal
  symbol.  Then for any $t\in \R$ and $x \in M$,
  \begin{xalignat}{2} \label{eq:leading_order_prop_quant}
\begin{split}     U_{k,t} ( \phi_t (x) , x) & = \Bigl( \frac{k}{2\pi} \Bigr)^n
 \bigl[ \rho_t (x) \bigr]^{\frac{1}{2}} \,  e^{ \frac{1}{i}
 \int_0^{t} H^{\s}_{r} ( \phi_r(x)) \; dr} \Bigr[ \phi_{t}^L (x) \Bigl]^{\otimes k} \otimes \mathcal{T}_{t}^{L'}
(x) \\ & + \bigo ( k^{n-1}) 
\end{split}
\end{xalignat}
where $\phi_t$ is the Hamiltonian flow of the principal symbol $H_t$,
$\phi_t^L$ and $\mathcal{T}_t^{L'}$ are its prequantum and parallel transport lifts, $H_t^{\s}$ is the subprincipal symbol and $(\rho_t)^{1/2}$ is
the continuous square root equal to $1$ at $t=0$ of the function $\rho_t$
such that $ \mathcal{D}_t
= \rho_t \mathcal{T}_t^K $ with $\mathcal{D}_t(x) = K( T_x \phi_t)$.

If $y \in M$ is different from $\phi_t (x)$, then $U_{k,t} ( y , x) = \bigo (
k^{-N})$ for all $N$.
\end{thm}

The first part of the result has an alternative formulation when $M$ has a half-form bundle, that is a line bundle $\delta$ and an
isomorphism between $\delta^2$ and the canonical bundle $K$. Introducing the
line bundle $L_1$ such that  $L'
= L_1 \otimes \delta $ and using that $\mathcal{T}^{L'}_t =
\mathcal{T}^{L_1}_t \otimes \mathcal{T}^{\delta}_t$, we obtain 
$$ U_{k,t} ( \phi_t (x) , x) \sim \Bigl( \frac{k}{2\pi} \Bigr)^n
 e^{ \frac{1}{i}
 \int_0^{t} H^{\s}_{r} ( \phi_r(x)) \; dr} \Bigr[ \phi_{t}^L (x) \Bigl]^{\otimes k} \otimes \mathcal{T}_{t}^{L_1}
(x)  \otimes  \bigl[ \mathcal{D}_t (x) \bigr]^{\frac{1}{2}} 
$$
where $\bigl[ \mathcal{D}_t (x) \bigr]^{1/2}: \delta_x \rightarrow
\delta_{\phi_t(x)}$ is the continuous square root of $\mathcal{D}_t(x)$ equal
to $1$ at $t=0$. Observe that to write this equation, it is sufficient to
define $\delta$ on the trajectory $t \rightarrow \phi_t(x)$, which is always
possible. 

In the above statement, we focused on the geometrical description of the
leading order term, because it is the real novelty.  The complete result, Theorem \ref{th:quantum_propagator}, too long
for the introduction, is that $U_{k,t} (\phi_t (x), x)$ has a full asymptotic
expansion in integral powers of $k^{-1}$ and we also have a uniform description
with respect to $x$, $y$ and $t$ on compact regions. Such a uniform description is not
obvious because the asymptotic behavior of $U_{k,t} ( y,x)$ is completely different
whether $y= \phi_t(x) $ or not. We actually show that the Schwartz kernel
of $U_{k,t}$ is
a Lagrangian state in the sense of \cite{C2003} (see the definition in Section \ref{sect:lag_fam}), associated with the graph of
$\phi_t$. 

In the Appendix, we investigate an explicit example in which the Hamiltonian flow does not preserve the complex structure, and verify the validity of the above theorem for the kernel of the propagator on the graph of this flow.

\subsection{Smoothed spectral projector} 

Our second result is the asymptotic description of the Schwartz kernel of $f(k ( E - T_k))$
where $(T_k)$ is a self-adjoint Toeplitz operator, $E$ is a regular value
of the principal symbol $H$ of $(T_k)$ and $f \in \Ci (\R, \R)$ is a smooth function
having a compactly supported Fourier transform. For a function $g: \R
\rightarrow \R$, $g(T_k)$ is merely defined as $\sum_{\la \in \op{sp}(T_k)} g(\la)
\Pi_{\la}$ where for each eigenvalue $\la$ of $T_k$, $\Pi_\la$ is the
orthogonal projector onto the corresponding eigenspace. For $g$ smooth,
$g(T_k)$ is itself a Toeplitz operator with principal symbol $g\circ H$ and
so its Schwartz kernel is concentrated on the diagonal; more precisely
\begin{xalignat*}{2}
g(T_k) (x,x ) & = \Bigl( \frac{k}{2 \pi } \Bigr)^n g(H(x)) + \bigo ( k^{n-1}),
\\ g(T_k) (x,y) & = \bigo ( k^{-N}), \quad \forall N  \qquad \text{ when } x \neq y.
\end{xalignat*}
In the rest of the paper we will work with a function $g$ depending on $k$ in
 the very specific way $g( \tau ) = f( k (\tau - E))$, which we interpret as
 a focus at scale $k^{-1}$ around $E$. For instance, for $f$ the characteristic
 function of a subset $A$ of $\R$, $f(k(\cdot - E))$ is the characteristic function of
 $E+ k^{-1} A$. However, we will only consider very regular functions $f$,
 having a smooth compactly supported Fourier transform $\hat{f}$.

Our result is that the Schwartz kernel of $f(k(E - T_k))$
 is (up to normalization by some power of $k$, see Remark \ref{rmk:normalisation} for a discussion) a Lagrangian state
associated with the Lagrangian immersion
\begin{gather} \label{eq:lag_immersion}
  j_E: \R \times H^{-1}(E) \rightarrow
M^2, \qquad (t, x) \rightarrow ( \phi_t (x) , x).
\end{gather}
Here $\phi_t$ is the flow of the autonomous Hamiltonian $H$. It is important to note that $j_E$
is not injective and not proper in general. However only the times $t$ in the
support of the Fourier transform of $f$ matter, so we will work with the
restriction of $j_E$ to a compact subset. Still, it is possible for $j_E$ to have
 multiple points because of the periodic trajectories.

The description of the Schwartz
kernel on the image of $j_E$ will be in terms of the parallel transport  lift
of $\phi_t$ to $L$ and $L'$ as introduced above and a lift
$\mathcal{D}'_t$ to the canonical bundle of the restriction of $\phi_t$ to the
energy level set $H^{-1}(E)$, defined as follows.

First since $E$ is
regular, for any $x \in H^{-1} ( E)$, the Hamiltonian vector field $X$ of
$H$ is not zero at $x$. Second, $H$ being time-independent, $T_x \phi_t$ sends
$X_x$ into $X_{ \phi_t (x)}$, so it induces a symplectic map from $T_x
H^{-1} (E) / \R X_x $ into  $T_{\phi_t(x)}
H^{-1} (E) / \R X_{\phi_t(x)} $. In the case $x$ is periodic with period
$t$, this map is the tangent linear map to the Poincar\'e section   map.

For any $ x\in H^{-1}(E)$, write $T_xM = F_x \oplus G_x$ where $F_x$ is the subspace
spanned by $X_x$ and $j_xX_x$ and $G_x$ is its symplectic orthogonal.
Observe that $T_x H^{-1}(E) = G_x \oplus \R X_x$ so that $G_x = T_x H^{-1}(E) / \R X_x$.
Furthermore,  both
$F_x$ and $G_x$ are symplectic subspaces preserved by $j_x$, so $K_x  \simeq
K(F_x) \otimes K(G_x)$.
Then we define
\begin{gather} \label{eq:defD'}
  \mathcal{D}'_t(x):K_x \rightarrow K_{\phi_t(x)}  
\end{gather}
as the tensor product of the following
maps:
\begin{enumerate}
\item $K(F_x) \rightarrow K(F_{\phi_t(x)})$,   $\la
  \mapsto 2 \| X_x \|^{-2} \la'$ where $\la$, $\la'$ are normalised by
  $\la (X_x) = \la' (X_{ \phi_t (x)}) = 1$,
  \item$ K(\psi) :K(G_x)  \rightarrow K(G_{\phi_t(x)}) $ with $\psi$ the symplectomorphism
    $$\psi: G_x \simeq T_x H^{-1}(E) / \R X_x \xrightarrow{T_x \phi_t}
    T_{\phi_t(x)} H^{-1}(E) / \R X_{\phi_t(x)} \simeq
    G_{\phi_t(x)}. $$
\end{enumerate}
In the particular case where $T_x \phi_t $ sends $(jX)_x$ into $
(jX)_{\phi_t(x)}$, one checks that $\mathcal{D}_t' (x) = 2 \| X_x\|^{-2}
\mathcal{D}_t (x)$. Otherwise, there does not seem to be any simple relation
between $\mathcal{D}_t(x)$ and $\mathcal{D}_t'(x)$.

Exactly as we did for $\mathcal{D}_t$, we define $\rho'_t (x)$ as the complex
number such that $${\mathcal{D}_t' (x)} = \rho'_t(x)  \mathcal{T}^K_{t}
(x).$$
We denote by $\bigl[ \rho'_t(x) \bigr]^{1/2}$ the continuous square root equal
to $\sqrt{2} \| X_x \|^{-1}$ at $t=0$.

\begin{thm} \label{th:intro2}
For any self-adjoint Toeplitz operator $(T_{k})$ and regular value $E$ of its
principal symbol $H$, we have for any $x, y \in H^{-1}(E)$,
\begin{multline*}
 f( k ( E - T_k)) (y,x) = \left(\frac{k}{2\pi}\right)^{n} k^{-\frac{1}{2}} \\
  \times \sum_{\substack{ t
  \in \op{Supp} \hat{f},\\  \phi_t (x) = y} } \hat{f} (t) \, \bigl[ \rho_t' (x) \bigr]^{\frac{1}{2}}
\,   e^{ \frac{1}{i}
 \int_0^{t} H^{\s} ( \phi_r(x)) \; dr} \, \Bigr[ \mathcal{T}_{t}^L (x) \Bigl]^{\otimes k} \otimes \mathcal{T}_{t}^{L'}
(x) + \bigo ( k^{n-\frac{3}{2}}).
\end{multline*}
Furthermore, for any $(x,y) \in M^2 $ not belonging to $ j_E( \op{Supp}(\hat{ f }) \times
H^{-1}(E))$, we have $  f( k ( E - T_k)) (x,y) = \bigo ( k^{-\infty})$. 
\end{thm}

As in Theorem \ref{th:intro1}, in the case $M$ has a half-form bundle $\delta$, we can
write  $L' = L_1 \otimes \delta$ and replace the sum above by
$$ \sum_{\substack{ t
  \in \op{Supp} \hat{f},\\  \phi_t (x) = y} } \hat{f}(t) \, e^{ \frac{1}{i}
 \int_0^{t} H^{\s} ( \phi_r(x)) \; dr} \, \Bigr[ \mathcal{T}_{t}^L (x) \Bigl]^{\otimes
 k} \otimes \mathcal{T}_{t}^{L_1} (x) \otimes \bigl[\mathcal{D}_t'(x)  \bigr]^{\frac{1}{2}}$$
where $ \bigl[ \mathcal{D}'_t (x)
\bigr]^{1/2} : \delta_x \rightarrow \delta_{\phi_t (x)} $ is continuous and
equal to $\sqrt{2} \ \| X_x \|^{-1} \op{id}_{\delta_x}$ at $t=0$.

Furthermore, we will give a uniform description with respect to $(x,y)$ of the Schwartz kernel of
$f( k (E - T_{k}))$ by showing it is a Lagrangian state associated with the Lagrangian immersion $j_E$.   

\subsection{Discussion}

Let us explain more the structure of the Lagrangian states appearing in the previous
results (see also Section \ref{sect:lag_fam} for precise definitions). Roughly, a Lagrangian state of $M$ is a family $(\Psi_k \in \Hilb_k,
\; k \in \N )$ which is $\bigo ( k^{-\infty})$ outside a Lagrangian submanifold $\Ga$
  of $M$ and  which has an asymptotic expansion at any point $x \in \Ga$ of the
  form
  \begin{equation} \Psi_k (x) = k^m [t (x)]^{\otimes k} \Bigl ( a_0 (x) + k^{-1} a_1 (x) +
  \ldots ) \label{eq:WKB} \end{equation}
  where $m$ is some real number, $t(x) \in L_x $ has norm one, and the coefficients $a_0(x)$, $ a_1 (x)$, $\ldots$ 
  belong to $L_x'$.  We can think of $[t(x)]^k$ as an oscillatory factor and
  $k^m \sum  k^{-\ell} a_{\ell} (x)$ as an amplitude, so the right-hand side of \eqref{eq:WKB} is completely analogous to the well-known WKB ansatz. Indeed, locally in a trivialization open set for $M$, $L$ and $L'$, sections of $L$ and $L'$ can be identified with functions, which yields $t(x)^{\otimes k} = e^{i k \phi(x)}$ for some phase $\phi$ which is real on $\Gamma$. 
 Furthermore, $t(x)$ and the $a_{\ell}(x)$ all depend smoothly
  on $x$ so that they define sections of $L$ and $L'$ respectively over $\Ga$. The
  section $t$ has the important property to be flat. Regarding the leading order term
  $a_0(x)$ of the amplitude, it is often useful to think about it as a
  product $t_1(x) \otimes \nu (x) $ where $t_1(x) \in (L_1)_x$ and $\nu (x)
  \in \delta_x$. Here $\delta$ is a half-form bundle, which can be introduced
  at least locally, and $L ' = L_1 \otimes \delta$. Then $[t(x)]^{\otimes k}
  \otimes t_1 (x)$
  may be viewed as a deformation of $[t(x)]^{\otimes k}$, whereas $\nu (x)$ is a square
  root of a volume element of $\Ga$. Indeed,  $\Ga$ being
  Lagrangian, there is a natural pairing between the restriction of the
  canonical bundle to $\Ga$ and the determinant bundle $\det T\Ga \otimes \C$.

  In our results, the Lagrangian states, which are Schwartz kernels of operators, are defined on $M^2$, with the prequantum bundle $L \boxtimes \con{L}$. Here, the
  symplectic and prequantum structures are such that the graphs of
  symplectomorphisms are Lagrangian submanifolds and their prequantum lifts
  define flat sections. 
In Theorem \ref{th:intro1}, the Schwartz kernel of the quantum propagator is
defined as a Lagrangian state associated with the graph of the Hamiltonian flow
and its prequantum lift. As was already noticed, the prequantum lift appears with
correction terms $\exp \bigl( i \int_0^t H_{r}^{\s} (\phi_r (x)) \, dr
 \bigr) $ and $\mathcal{T}_t^{L_1} (x)$, which are the contributions of the
 corrections $H^{\s}$ to $H$ and $L_1$ to $L$. Then the last term
 $\mathcal{D}_t^{\frac{1}{2}}$ is merely the
 square root of the image of the Liouville volume
 form by the map $M \rightarrow \op{graph} \phi_t$ sending $x$ into $(\phi_t
 (x),x )$. 

The relation between Theorems \ref{th:intro1} and \ref{th:intro2} relies on the time/energy
duality. Roughly, for a time-independent operator $\hat{H}$,
we pass from the quantum propagator $(
\exp ( -i\hbar^{-1} t \hat{H} ),\; t\in \R)$ to the smoothed spectral projector $( f(\hbar^{-1} ( \hat{H} - E)), \; E \in \R)$, by multiplying by $\hat f (t)$ and
then doing a partial $\hbar$-Fourier transform with respect to the variables
$t, E$ (here $k$ plays the part of $\hbar^{-1}$). In the microlocal point of view, the variables $t$ and $E$ are
equivalent and we can view the quantum propagator and the spectral projector
as two facets of the same object.

In our results, this duality is expressed by the fact that the graph of $\phi_t$ and the Lagrangian immersion $j_E$ are obtained in a
symmetric way from the Lagrangian submanifold
\begin{gather} \label{eq:grande_lagrangienne} 
\wt{\Ga} = \{ ( t, H(x), \phi_t (x),
x) / x\in M, t\in \R \}
\end{gather}
of $T^*\R \times M \times M^{-}$. Indeed, the graph of
$\phi_t$ and the image of $j_E$ are the projections onto $M^2$ of the
slices
$$ \wt{\Ga}_t = \wt{\Ga} \cap (\{ t \} \times \R \times M^2), \qquad \wt{\Ga}^E = \wt{\Ga} \cap (\R \times \{ E \} \times M^2).$$ 
The prequantum lifts and the volume elements can also be incorporated in this
picture. In particular, we pass from $\mathcal{D}_t$ to  $\mathcal{D}'_t$ by canonical isomorphisms between volume elements of $\wt{\Ga}$,
$\wt{\Ga}_t$ and $\wt{\Ga} ^E$. 

As we will see in a next paper, the quantum propagator viewed as a function of time is
actually a Lagrangian state associated with $\wt{\Ga}$ (and we will particularly focus on the computation of the precise geometric quantities involved in its principal symbol). This statement is
delicate because here we mix real and complex variables, cotangent bundles and
K\"ahler manifolds, and the description
of Lagrangian states is rather different in these two settings. To give a
sense to this, we will perform a Bargmann transform so that the quantum
propagator will become a holomorphic function of the complex variable $t +iE$.
This point of view will be interesting, even for the proof of
Theorem \ref{th:intro1}, to understand the transport equation satisfied by the
leading order term of the amplitude. 

\subsection{Comparison with earlier results}
\label{subsect:comparison}

The introduction of Fourier integral operators with application to the Schrö\-dinger
equation and spectral properties of pseudodifferential operators has its
origin in the seminal H\"ormander \cite{Hor1} and Duistermaat-Guillemin 
\cite{DuGu} papers, cf. the
survey \cite{Gu_FIO}. In these first developments, the operator under study is the
Laplace-Beltrami operator and the corresponding classical flow is the geodesic
flow.

The transcription of these results to Berezin-Toeplitz operators has been done
in the paper \cite{BPU} by Bothwick-Paul-Uribe, by applying the Boutet de
Monvel-Guillemin approach of \cite{BoGu}. Similar results have been proved in
a recent paper \cite{ZeZh} by Zelditch-Zhou 
where the application to spectral densities has been pushed further.
These papers both rely on the Boutet de Monvel-Guillemin book \cite{BoGu}. In particular
the properties of Berezin-Toeplitz operators are deduced from the
pseudodifferential calculus, and the quantum propagator is viewed as a
Fourier integral operator. From this is deduced the asymptotics of the smoothed spectral
projector on the diagonal, \cite[Theorem 1.1]{BPU} and \cite[Theorem
2.2]{ZeZh}. The leading order term is computed in terms of the symbolic calculus
of Fourier integral operator of Hermite type in \cite{BPU}, or with a
non-linear problem in Bargmann space in \cite{ZeZh}. Another description of the kernel of the quantum propagator associated with an autonomous Hamiltonian was obtained by Ioos \cite{Ioos}; this description involves quantities related with parallel transport in the canonical bundle with respect to a connection induced by the tranport of the initial complex structure by the Hamiltonian flow, and computing these coefficients appears to be quite complicated in general, relatively to our formulas. In fact, in all these works the analysis is well-understood but our main addition, apart from obtaining a direct derivation in our context, is the precise computation of the geometric quantities contained in the principal symbol of the kernel of the quantum propagator seen as a Lagrangian state. In particular these quantities have a very natural interpretation in terms of half-forms, and can be easily computed for concrete examples.

The techniques that we use come from the work of the first author  where a direct definition
of Lagrangian states on a K\"ahler manifold is introduced \cite{C2003}.
As explained in the discussion following Equation \eqref{eq:WKB}, these
  Lagrangian states locally look like WKB functions with complex phase. The
  microlocal toolbox for complex phase WKB states was developed in the seminal
  paper \cite{MeSj} in the homogeneous case. However our Lagrangian states are
  specific to the K\"ahler setting, for instance, the states being defined directly on phase
  space, there are no caustics. Moreover the relevant symplectic geometry is
  not the geometry of the cotangent space of the base but the K\"ahler geometry
  of the base itself.

In the first author's PhD thesis \cite[Section 3.5.2]{ChaPhD}, it is shown that the quantum propagator is a Lagrangian state, but without the precise computation of the principal symbol that we obtain here. The use of half-form bundles for Berezin-Toeplitz operators started in \cite{C2006, C2007} and here we apply them to the description of the quantum propagator. The
isomorphisms \eqref{eq:defkg} have been introduced in \cite{C2007},
\cite{C2010} where their square roots
are called half-form bundle morphisms. A similar invariant appears in
\cite{ZeZh} under the form \eqref{eq:pol_formula}.
Again, we insist that the main novelty in our results
is the precise description and computation of the coefficients $\rho_t(x)$ and $\rho'_t(x)$
appearing in Theorems \ref{th:intro1}, \ref{th:intro2}.

Whereas the relation between the
quantum propagator and the Hamiltonian flow is a classical result, the similar statement
for the
smoothed spectral projector and the Lagrangian immersion
\eqref{eq:lag_immersion} seems to be new.  In \cite{BPU} and
\cite{ZeZh}, only the diagonal behavior of the Schwartz kernel is described.
To state our result, we will introduce a general class of Lagrangian states associated with Lagrangian
immersions.

\subsection{Outline of the paper} 

Section \ref{sect:lag_fam} is devoted to time-dependent Lagrangian states, that we call
Lagrangian state families. The main result is that these states provide solutions to the
Schr\"odinger equation with quantum Hamiltonian a Toeplitz operator and initial data a Lagrangian state, cf. Theorem
\ref{th:Lagrangian_State_propagation}. The principal symbol of these solutions
satisfies a transport equation, that is solved in Section
\ref{sec:transport-equation} (up to a rather technical part which is postponed to Section \ref{sec:proof-theo} for the sake of clarity),  while in Section \ref{sect:meta}, we give an
elegant expression in the context of metaplectic
quantization.
These results will be applied in Section \ref{sect:propagator} to the quantum
propagator, where Theorem \ref{th:intro1} is proved.

In Section \ref{sect:imm_lag}, we prove that the Fourier transform of a time-dependent Lagrangian
state is a Lagrangian state as well, with an underlying Lagrangian manifold
which is in general only immersed and not embedded, cf. Theorem \ref{th:Lagrangian_Fourier}. The needed adaptations in
the Lagrangian state definition for immersed manifolds are given in Section
\ref{sec:immers-lagr-stat}. 
In Section \ref{sect:proj}, we deduce Theorem \ref{th:intro2} on the smoothed
spectral projector. 

\paragraph{Acknowledgments.} We thank two anonymous referees for useful comments.

\section{Propagation of Lagrangian states} 
\label{sect:lag_fam}

In this section, we introduce some one-parameter families of Lagrangian
  states which are relevant to our setting and study how they evolve under the
  Schr\"odinger equation. The definition of these states is new and builds on
  the standard definition of Lagrangian states introduced in \cite{C2003},
  which we briefly recall now. It will also be useful to have the standard definition in mind when introducing Lagrangian states associated with immersed Lagrangians, see Section \ref{sec:immers-lagr-stat}.

Let $M$, $L$ and $L'$ be as in Section \ref{subsect:BTO}. Let $\Gamma$ be a Lagrangian submanifold of $M$ equipped with a flat unitary section $s \in \mathcal{C}^{\infty}(\Gamma,L)$. A \emph{Lagrangian state} associated with $(\Gamma,s)$ is a sequence $(\Psi_k \in \Hil_k)_{k \geq 1}$ of the form
\[ \Psi_k(x) = \left( \frac{k}{2\pi}\right)^{\frac{n}{4}} F^k(x) a(x,k) + \bigo(k^{-\infty}) \]
where 
\begin{itemize}
\item[-] $F \in \mathcal{C}^{\infty}(M,L)$ is such that $\bar{\partial} F$ vanishes to infinite order along $\Gamma$,
\item[-]  $F_{| \Gamma} = s$ and $|F(x)| < 1$ for $x \notin \Gamma$,
\item[-]  $a(\cdot,k)$ is a sequence of smooth sections of $L' \to M$ with an asymptotic expansion $a(\cdot,k) = \sum_{\ell \geq 0} k^{-\ell} a_{\ell}$ for the $\mathcal{C}^{\infty}$ topology, where each section $a_{\ell}$, for $\ell \geq 0$, is such that $\bar{\partial} a_{\ell}$ vanishes to infinite order along $\Gamma$,
\item[-]  the $\bigo$ is for the pointwise norm and uniform on $M$.
\end{itemize}

For any sequence $(b_{\ell})_{\ell \geq 0}$ of elements of
$\mathcal{C}^{\infty}(\Gamma,L')$, there exists a Lagrangian state $\Psi_k$
such that for every $\ell \geq 0$, $b_{\ell} = {a_{\ell}}_{|\Gamma}$.
 The \emph{full symbol} of $\Psi_k$ is the formal series $\sum_{\ell \geq 0} \hbar^{\ell} b_{\ell}$, and uniquely determines $\Psi_k$ up to $\bigo(k^{-\infty})$. The first term $b_0$ in this full symbol is called the \emph{principal symbol} of $\Psi_k$.

Since we will use later some generalisations of this construction, let us
briefly recall the proof, the details being in \cite[Section 2]{C2003}. First,
since $\Gamma$ is a totally real submanifold, any smooth function of
$\Gamma$ has an extension $f$ to $M$ such that $\con{\partial} f$ vanishes to
infinite order along $\Ga$. The same holds for the sections of a holomorphic line
bundle. In this way we construct $F$ and the $a_{\ell}$'s from  $s$ and the
$b_{\ell}$'s respectively. These sections are not uniquely determined, but their Taylor
expansion along $\Ga$ is. In particular, a computation shows that $\ln |F|$
has a non degenerate minimum along $\Ga$, so modifying $F$ away from $\Ga$ if
necessary, the condition $|F|<1 $ on $M \setminus \Ga$ is satisfied. The
Lagrangian state $\Psi_k$ is then obtained by projecting the smooth section $\widetilde{\Psi}_k = ( k/2 \pi)^{n/4}F^k
a( \cdot, k)$ onto $\Hilb_k$. We claim that $\Psi_k = \widetilde{\Psi}_k + \bigo (
k^{-\infty})$. The proof of this fact was obtained by stationary phase computations in \cite{C2003}.
Alternatively, this follows from the fact that $\con{\partial} \Psi_k = \bigo
( k^{-\infty})$ and the Kodaira-Hörmander estimates \cite{Kod,Hor_dbar}.

\begin{rmk} 
\label{rmk:normalisation}
The normalization factor $\bigl( \frac{k}{2\pi}\bigr)^{\frac{n}{4}}$ is
somewhat arbitrary. First, the power of $2\pi$ could be included in the symbol
of the Lagrangian state. Second, the choice of the power of $k$ is more or
less convenient depending on the context, since Lagrangian states appear in
different situations (for instance as approximate eigenvectors for
Berezin-Toeplitz operators, or as an ansatz for the Schwartz kernel of such an
operator).  Here the choice of normalization yields a $L^2$-norm of order
$\bigo(1)$ for the Lagrangian states. \qed
\end{rmk}

\subsection{Families of Lagrangian states} 
\label{sec:famil-lagr-stat}

As explained above, to define a Lagrangian state, we need a Lagrangian submanifold  of $M$
equipped with a flat unitary section of the prequantum bundle $L$. Let us
consider a one-parameter family of such pairs. More precisely, let $I \subset \R$ be an
open interval, $\C_I$ be the trivial complex line bundle over $I$, $\Ga$ be a closed submanifold of $I \times M$ and $s \in \Ci ( \Ga, \C_I \boxtimes L )$ be such that
\begin{enumerate}
\item  the map $q: \Ga \rightarrow I$, $q(t,x) = t$ is a proper submersion. So for any $t \in I$, the fiber $\Ga_t := \Ga \cap ( \{ t \} \times M)$ is a submanifold of $M$,
\item for any $t \in I$, $\Ga_t$ is a Lagrangian submanifold of $M$ and the restriction of $s$ to $\Ga_t$ is flat and unitary.
\end{enumerate}

\begin{rmk} \begin{enumerate} 
\item[a.]  Since $q$ is a proper submersion, by Ehresmann's lemma, $\Ga$ is
  diffeomorphic to $I \times \mathcal{N}$ for some manifold $\mathcal{N}$ in such a way that  $q$
  becomes the projection onto $I$.
  
\item[b.] Given a proper submersion $q: \Ga \rightarrow I$ and a map $f : \Ga \rightarrow M$, it is equivalent that the map $\Ga \rightarrow I\times M$, $ x \rightarrow (q(x), f(x))$ is a proper embedding  and that for any $t$, $f(t,\cdot): \Ga_t \rightarrow M$ is an embedding. We decided to start from a closed submanifold of $I \times M$ to be more efficient in the definition of Lagrangian states below. \qed
  \end{enumerate} 
\end{rmk}

We will consider states $\Psi_k$ in $\mathcal{H}_k$ depending smoothly on $t \in I$, so $\Psi_k$ belongs to $\Ci ( I , \mathcal{H}_k)$. Equivalently $\Psi_k$ is a smooth section of $\C_I \boxtimes( L^k \otimes L')$ such that $\con{\partial}\Psi_k = 0$. Here the $\con{\partial}$ operator only acts on the $M$ factor. Similarly, it makes sense to differentiate with respect to $t \in I$ a section of $\C_I \boxtimes A$, where $A$ is any vector bundle over $M$. 

\begin{dfn} A {\em Lagrangian state family} associated with $( \Ga, s)$ is a family $( \Psi_k \in \Ci (I, \mathcal{H}_k)$, $k \in \N)$ such that for any $N$, 
\begin{gather} \label{eq:def_lagrangian} 
\Psi_k ( t,x) = \Bigl( \frac{k}{2\pi} \Bigr)^{\frac{n}{4}} F^k(t,x) \sum_{\ell = 0 }
^{N} k^{-\ell} a_{\ell} (t,x) + R_N(t,x,k) 
\end{gather}
where
\begin{enumerate}
\item[-] $F$ is a section of $\C_I \boxtimes L$ such that $F|_\Ga = s$, $\con{\partial} F$ vanishes to infinite order along $\Ga$ and $|F| <1$ outside of $\Ga$,
\item[-] $(a_\ell)$ is a sequence of sections of $\C_I \boxtimes L'$ such that $\con{\partial} a_\ell$ vanishes to infinite order along $\Ga$,
\item[-] for any $p$ and $N$, $\partial^p_t R_N = \bigo ( k^{ p - N -1})$ in pointwise norm uniformly on any compact subset of $I \times M$.   
\end{enumerate}
\end{dfn}
It is not difficult to adapt the argument of \cite[Section 2]{C2003} and to prove the following facts. First, the section $F$ exists. Second, we can specify arbitrarily the coefficients $a_\ell$ of the asymptotic expansion along $\Ga$ and this determines $(\Psi_k)$ up to $\bigo ( k^{-\infty})$. More precisely, for any sequence $(b_\ell) \in \Ci ( \Ga, \C_I \boxtimes L')$, there exists a Lagrangian state $(\Psi_k)$ satisfying for any $y  \in \Ga$
$$ \Psi_k (y)  = \Bigl( \frac{k}{2\pi} \Bigr)^{\frac{n}{4}} s^k(y)  \sum_{\ell = 0 }^N k^{- \ell} b_\ell(y)  + \bigo ( k^{-N - 1}) \qquad \forall N .$$
Furthermore, $(\Psi_k)$ is unique up to a family $(\Phi_k \in \Ci (I,
\mathcal{H}_k), \; k \in \N)$ satisfying $\| (\frac{d}{dt})^p \Phi_k (t)\| =
\bigo ( k^{ - N})$ for any $p$ and $N$ uniformly on any compact subset of
$I$.

We will call the formal series $\sum \hbar^{\ell} b_{\ell} $ the {\em full
  symbol} of $(\Psi_k)$. The first coefficient $b_0$ will be called the {\em
  principal symbol}.

It could be interesting to define Lagrangian state families with a different
regularity with respect to $t$. Here, our ultimate goal is to solve a Cauchy
problem, so we will differentiate with respect to $t$ and in a later proof, we will
use that $(k^{-1} \partial_t \Psi_k )$ is still a Lagrangian state family. So we need to consider states which are smooth in $t$.
Observe that in the estimate satisfied by $R_N$ we lose one power of $k$ for
each derivative; this is consistent with the fact that $
\partial_t (F^k) = k F^k f$ where $f$ is the logarithmic derivative of $F$, that is $\partial_t F = f F$. 

A last result which is an easy adaptation of \cite[section 2.4]{C2003} is the action of
Toeplitz operators on Lagrangian state families. Let $(T_{k,t})$ be a time-dependent Toeplitz operator and
$(\Psi_{k,t})$ be a Lagrangian
state as above. Then $(T_{k,t}\Psi_{k,t})$  is a Lagrangian
state family associated with $( \Ga, s)$ as well. Furthermore, its full symbol is equal to $\sum
\hbar^{\ell+m} Q_m (b_{\ell})$ where $\sum  \hbar^{\ell} b_{\ell}$ is the full
symbol of $(\Psi_k)$ and the $Q_m$ are differential operators acting on
$\Ci ( \Ga, \C_I \otimes L')$ and depending only on $(T_{k,t})$. In particular, $Q_0$ is 
the multiplication by the restriction to $\Ga$ of the principal symbol of
$(T_k)$.

\subsection{Propagation} 

Consider the same data $\Ga \subset I \times M$ and $s\in \Ci ( \Ga, \C_I
\boxtimes L)$ as in the previous section.
We claim that the covariant derivative of $s$ has the form
\begin{gather} \label{eq:def_tau}
\nabla s = i \tau dt  \otimes s \qquad \text{with} \qquad \tau \in \Ci (
\Ga, \R).
\end{gather}
Here the covariant derivative is induced by the trivial derivative of
$\C_{I}$ and the connection of $L$. To prove \eqref{eq:def_tau}, use
that the restriction of $s$ to each $\Ga_t$ is
flat. So $\nabla s = i \al \otimes s$, with $\al \in \Om ^ 1 ( \Ga , \R )$
vanishing in the vertical directions of $q: \Ga \rightarrow I$. So $\al =
\tau dt$ for some function $\tau \in \Ci ( \Ga, \R)$.

In the following two propositions,  $(\Psi_k)$ is a Lagrangian state family associated with $(\Ga, s)$ with full symbol
$b(\hbar) = \sum \hbar^\ell b_{\ell}$. 

\begin{prop} \label{prop:derive_etat_lag}
 $( (ik)^{-1} \partial_t \Psi_k)$ is a
Lagrangian state family associated with $(\Ga, s)$ with full symbol $(\tau +
\hbar P) b( \hbar)$, where $P$ is a differential operator of $\Ci ( \Ga,
\C_{I} \boxtimes L')$. 
\end{prop}

\begin{proof}
  Differentiating the formula \eqref{eq:def_lagrangian} with respect to $t$, we obtain on a
  neighborhood of $\Ga$ that
  $$  \partial_t \Psi_k  = \Bigl( \frac{k}{2\pi} \Bigr)^{\frac{n}{4}} F^k
   \sum_{\ell } k^{-\ell} ( k f a_\ell  + \partial_t a_\ell )  $$
  where $f \in \Ci ( I \times M)$ is the logarithmic derivative of $F$ with
  respect to time, so 
  $\partial_t F = f F$. Using that $\con{\partial}$ and the derivative with respect to
  $t$ commute, we easily prove that  $\con{\partial}f $ and
  $\con \partial (\partial_t a_\ell) $ both vanish to infinite order
  along $\Ga$. This shows that $(k^{-1} \partial_t \Psi_k)$ is a Lagrangian state
  associated with $(\Ga, s)$.

  Its full symbol is the restriction to $\Ga$ of the
  series $ \sum \hbar^{\ell} (f a_{\ell } + \hbar \partial_t a_{\ell})$. We claim that
  $f |_{\Ga} = i \tau$. Indeed, at any point of $\Ga$, $\nabla F$ vanishes in
  the directions tangent to $M$, because it vanishes in the directions of type
  $(0,1)$ and in the directions tangent to $\Ga_t$ as well. So $\nabla F = f dt \otimes
  F$ along $\Ga$. The restriction of $F$ to $\Ga$ being $s$,
  \eqref{eq:def_tau}
   implies that $f |_{\Ga} = i \tau$.

  Using similarly that at any $x \in \Ga_t$, $(T_x \Ga_t\otimes
  \C ) \oplus T^{0,1} _x M = T_x M \otimes \C$, and $\con \partial a_{\ell} =0$ along
  $\Ga$, it comes that $\partial_t a_{\ell} = \nabla_Z
  a_{\ell}$ along $\Ga$, where  $Z(t,x) \in T_{(t,x)} \Ga$ is the projection of
  $\partial/\partial t$ onto $T_{(t,x)} \Ga$ parallel to $ T^{0,1}_xM$. This
  concludes the proof with $P$ the operator $\frac{1}{i} \nabla_Z$.  
\end{proof}
  
We now assume that $(\Ga, s)$ is obtained by propagating a Lagrangian
submanifold $\Ga_0$ of $M$ and a flat section $s_0$ of $L \rightarrow \Ga_0$
by a Hamiltonian flow and its prequantum lift. Let $(H_{t})$ be the
time-dependent Hamiltonian generating our flow $( \phi_t)$ and denote by
$\phi_t^L$ its prequantum lift defined as in the introduction by
\eqref{eq:prequantum_lift}. So we set
$$ \Ga_t = \phi_t ( \Ga_0), \qquad s_t (
\phi_t (x) ) = \phi_t^L (x) s_0 (x).$$ 
Let $Y$ be  the vector field  of $\R \times M$ given by $ Y(t,x) =
\frac{\partial}{\partial t}  + X_t (x)$ where $X_t$ is the Hamiltonian vector field of
$H_{t}$.

Introduce a time-dependent Toeplitz operator $(T_{k,t})$ with principal symbol
$(H_{t})$.
\begin{prop}
\label{prop:rho_transport} $(\frac{1}{ik} \partial_t \Psi_k + T_{k,t} \Psi_k)$ is a
  Lagrangian state family
  associated with $(\Ga, s)$ with full symbol $\hbar (\frac{1}{i}\nabla_Y +
  \zeta )b_0 +
  \bigo( \hbar^2)$ for some $\zeta \in \Ci ( \Ga)$.  
\end{prop}

\begin{proof} By Proposition \ref{prop:derive_etat_lag} and the last paragraph
  of section \ref{sec:famil-lagr-stat}, we already know that $(\frac{1}{ik} \partial_t\Psi_k + T_{k,t} \Psi_k)$ is a
  Lagrangian state with full symbol
  \begin{gather} \label{eq:symb_total_prop}
    (\tau(t,x) + H_t(x)) (b_0 + \hbar b_1) + \hbar Q
  b_0 + \bigo (\hbar^2),
\end{gather}
where $Q$ is a differential operator acting on $\Ci ( \Ga, \C_I \otimes
  L')$. By differentiating \eqref{eq:prequantum_lift} in the definition of $s_t$ and by the fact that $\nabla s = i \tau dt \otimes s$, it comes that
\begin{gather} \label{eq:function_tau}
  \tau (t,\phi_t(x)) + H_{t}  (\phi_t(x)) =0,
\end{gather}
so the leading order term in
 \eqref{eq:symb_total_prop} is zero.
  Consider $f\in \Ci ( \R \times M)$ and compute the commutator
  \begin{xalignat*}{2} \bigl[ \tfrac{1}{ik} \partial_t +
    T_{k,t} , T_k (f) \bigr] & =
  \tfrac{1}{ik} \bigl( T_k( \partial_t f) + T_k ( \{ H_t, f \} \bigr) + \bigo ( k^{-2})
  \\ & =
  \tfrac{1}{ik} T_k ( Yf) + \bigo ( k^{-2}).
\end{xalignat*}   
Letting this act on our Lagrangian state family $\Psi_k$, it comes that $[Q,
f|_\Ga] = \frac{1}{i} (Yf)|_{\Ga}$. Since this holds for any $f$, this proves that $iQ$ is a
derivation in the direction of $Y$ so $iQ = \nabla_Y + i \zeta$ for some function
$\zeta$. 
\end{proof}

\begin{thm} \label{th:Lagrangian_State_propagation}
  For any Lagrangian state $(\Psi_{0,k} \in \mathcal{H}_k)$ associated with
  $(\Ga_0, s_0)$, the solution of the Schr\"odinger equation
\begin{gather} \label{eq:Schrod_lagrangian}
  \tfrac{1}{ik} \partial_t \Psi_k + T_{k,t} \Psi_k = 0, \qquad \Psi_k(0, \cdot) =
  \Psi_{0,k} 
\end{gather}
is a Lagrangian state family associated with $(\Ga, s)$ with symbol $b_0 + \bigo (
  \hbar)$ where $b_0$ satisfies the transport equation $\frac{1}{i} \nabla_Y b_0 + \zeta b_0
  = 0$.
\end{thm}

Since the integral curves of $Y$ are $t \mapsto ( t , \phi_t(x))$, the
solution of the transport equation is 
\begin{gather} \label{eq:sol_transp}
b_0 (t, \phi_t (x) ) =  e^{\frac{1}{i} \int_0^t \zeta ( r, \phi_r(x)) \,
  dr } \mathcal{T}^{L'}_t (x) \, b_0 ( 0,x) .
\end{gather}
In the next section, we will give a geometric formula for $\zeta$ in terms of
the canonical bundle.  

\begin{proof} The proof is the same as for differential operators (see the proof of Theorem 20.1 in \cite{Shu} for instance), so we only
  sketch it. 
We successively construct the coefficients $b_\ell$ to solve $ \frac{1}{ik}
\partial_t \Psi_k
+ T_{k,t} \Psi_k = \bigo ( k^{-N-1})$ with initial condition  $\Psi_k(0, \cdot) =
  \Psi_{0,k} + \bigo ( k^{-N})$. At each step, we have to solve a transport
  equation $ \nabla_Y b_N + d b_N = r_N $ with initial condition $b_N ( 0,
  \cdot) = b_{0,N}$, which has a unique solution. This provides us with a
  Lagrangian state $(\Psi_k)$ such that both equations of \eqref{eq:Schrod_lagrangian}
  are satisfied up to a $\bigo ( k^{-\infty})$. Then, applying Duhamel's
  principle, we show that the difference
  between $\partial_t\Psi_k $ and the actual solution of \eqref{eq:Schrod_lagrangian}
  is a $\bigo ( k^{-\infty})$ uniformly on any bounded interval.
\end{proof}

\subsection{Transport equation} \label{sec:transport-equation}

We will now give a formula for the function $\zeta$ and solve the above transport equation.
Essential to our presentation are line bundle isomorphisms involving
the canonical bundle $K$ of $M$ and the determinant bundles $\bigwedge^n
T^*\Ga_t$ and $\bigwedge^{n+1} T^* \Ga$.

First, for any $t \in I$, let $K_t$ be the restriction of $K$ to $\Ga_t$. Then
we have an isomorphism
\begin{gather}\label{eq:iso_Kt_volume}
  K_t \simeq \det (T^*\Ga_t) \otimes \C
\end{gather}
defined by sending $\Om \in (K_t)_x = \bigwedge^{n, 0
} T_x^*M$ to its restriction to $T_x \Ga_t \subset T_xM$. This is an
isomorphism because $(T_x \Ga_t \otimes \C)  \cap T_x^{0,1}M = \{ 0 \} $, which follows from
the fact that $\Ga_t$ is Lagrangian. 

Second, $\Ga_t$ being a fiber of $\Ga \rightarrow I$, the linear tangent maps
to the injection $\Ga_t \rightarrow \Ga$ and the projection $\Ga \rightarrow \R$
give an exact
sequence
$$0 \rightarrow T_x \Ga_t \rightarrow T_{(t,x)} \Ga \rightarrow \R = T_t^*I
\rightarrow 0 .$$
Since $\R$ has a canonical volume element, we obtain an isomorphism
\begin{gather} \label{eq:iso_vol}
  \det(T^* \Ga_t) \simeq \det (T^* \Ga)|_{\Ga_t}
\end{gather}
defined in the usual way: for any $\al \in \bigwedge^n T^*_{(t,x)} \Ga$, one
sends $dt \wedge \al \in \bigwedge^{n+1} T^* _{(t,x)} \Ga $ into the
restriction of $\al$ to $T_x\Ga_t$.

Gathering these two isomorphisms , we get a third one:
\begin{gather}  \label{eq:iso_K_det_Gamma} 
K_{\Ga}  := ( \C_I \boxtimes K )|_{\Ga} \xrightarrow{\simeq} \det(T^*\Ga ) \otimes \C,
\qquad (1 \boxtimes \al )|_{\Ga} \mapsto j^* (dt \wedge \al)    
\end{gather}
for any $\al \in \Om^{n,0}(M)$ with $j$ the embedding $\Ga \rightarrow I
\times M$.  
On the one hand, $K_{\Ga}$ has a natural connection induced by the Chern connection of $K$,
which gives us a derivation $\nabla_Y$ acting on sections of $K_\Ga$. On the
other hand, the Lie derivative $\mathcal{L}_Y$ acts on the differential forms of $\Ga$, and
in particular on the sections of $\det (T^* \Ga)$. Under the isomorphism \eqref{eq:iso_K_det_Gamma},
$$\mathcal{L}_Y = \nabla_Y + i \theta$$ where $\theta \in \Ci ( \Ga)$ since $\mathcal{L}_Y$ and $\nabla_Y$ are
derivatives in the same direction $Y$.

\begin{thm} \label{th:gamma}
The function $\zeta$ defined in Proposition \ref{prop:rho_transport} satisfies the equality $\zeta = \frac{1}{2} \theta + H^{\s}|_\Ga$. 
\end{thm}

The proof is postponed to Section \ref{sec:proof-theo} since it does not help to
understand what follows and it is quite technical. On the one hand, we can
compute $\theta$ in terms of second derivatives of $H_t$, cf. Proposition \ref{prop:delta}. On the other hand,
we directly compute the function $\zeta$, cf. Proposition
\ref{prop:Lag_state_symbol_calculus}.

We will now give an explicit expression for the term involving $\zeta$ in the
solution \eqref{eq:sol_transp} of the transport equation in light of Theorem \ref{th:gamma}.
For any $t \in I$, the tangent map to $\phi_t$ restricts to an isomorphism from $T\Ga_0$
to $T\Ga_t$. By the identification \eqref{eq:iso_Kt_volume},
we get an isomorphism $\mathcal{E}_t$ from $K|_{\Gamma_0}$ to $K|_{\Gamma_t}$
lifting $\phi_t$. More precisely,  for any $x\in \Gamma_0$, $u \in K_x$ and $v\in \op{det}
 (T_x \Ga_0)$, we define $\mathcal{E}_t(x) u \in K_{\phi_t(x)}$ so that  
\begin{gather} \label{eq:def_E_transport}
(\mathcal{E}_t(x) u )\bigl( (T_x
\phi_t)_* v \bigr) = u
(v) .
\end{gather}
The parallel transport $\mathcal{T}_t^K$ restricts as well to an isomorphism
$K|_{\Gamma_0} \rightarrow K|_{\Gamma_t}$. Define the complex number $C_t (x)$ by $ \mathcal{E}_t
(x)= C_t(x) \mathcal{T}_t^K (x)$.

\begin{prop} \label{prop:solution_transport_equation}
  The solution of the transport equation $\frac{1}{i} \nabla_Y b + \zeta b =0$ with $b
  \in \Ci ( \Ga, L')$ is
\begin{gather} \label{eq:solution_transport}
  b (t,\phi_t(x)) = C_t(x)^{\frac{1}{2}} \, e^{\frac{1}{i}
\int_0^t H_r^{\s} ( \phi_r (x)) \, dr  } \,   \mathcal{T}_t^{L'} (x) \, b(0,x)  
\end{gather}
with the square root of $C_t(x)$ chosen continuously and $C_0 =1$.
\end{prop}

\begin{proof}
In view of Equation \eqref{eq:sol_transp} and Theorem \ref{th:gamma}, it suffices to deal with the case $H^{\s}|_\Ga =0$. Moreover, observe that if $\tilde{b}$ satisfies $\nabla_Y \tilde{b} = 0$, then $b = f \tilde{b}$ solves $\frac{1}{i} \nabla_Y b + \zeta b = 0$ if and only if $\frac{1}{i} Y . f + \zeta f =0$. So it suffices to prove that $f: (t,\phi_t(x)) \mapsto C_t(x)^{1/2}$ is a solution of the latter equation.

First the isomorphism $I \times \Ga_0 \simeq \Ga$, $ (t,x) \rightarrow (t,
\phi_t(x))$ sends the vector field $\partial_t$ to $Y$. The solutions
of $\mathcal{L}_{\partial_t} \beta = 0 $ with $\beta \in \Om^{n+1} ( I \times \Ga_0)$ have the form $\beta = dt \wedge \beta_0$ with
$\beta_0 \in \Om^n ( \Ga_0)$. So the solutions of $\mathcal{L}_Y \al =0$ with
$\al \in \Om^{n+1}( \Ga)$ are parametrised by $\al_0 \in \Om^n ( \Ga_0)$ and
given by 
$$\al|_{(t,\phi_t(x))} = dt \wedge (\phi_t^*)^{-1} \al_{0}|_x.$$
Now, identify $K_{\Ga}$ and $\det(T^*\Ga) \otimes \C$ through
\eqref{eq:iso_K_det_Gamma}. Then by \eqref{eq:def_E_transport}, the previous equation becomes
$$\al|_{(t, \phi_t (x))} = \mathcal{E}_t (x) \, \al|_{(0,x)}.$$ 
Second, the solutions of $\nabla_Y \al' =0$ with now $\al' \in \Ci (\Ga, K_{\Ga})$ are given
by 
$$\al'|_{(t, \phi_t (x))} = \mathcal{T}^K_t (x) \, \al'|_{(0,x)}.$$ 
Assume that $\al'|_{(0,x)} = \al|_{(0,x)}$; then we have $\al
= C \al '$ with $C \in \Ci ( \Ga)$ defined by $C(t,\phi_t(x)) = C_t(x)$.  
Therefore
\[ 0 = \mathcal{L}_Y \al = \mathcal{L}_Y (C \al') = (Y . C) \al' + C \mathcal{L}_Y \al' = (Y . C) \al' + C \underbrace{\nabla_Y \al'}_{=0} + 2 i \zeta C \al' \]
so $Y.C + 2 i \zeta C =0$, hence $\frac{1}{i} Y . C^{1/2} + \zeta C^{1/2} =0$. 
\end{proof}

\section{Metaplectic correction} 
\label{sect:meta}

It is useful to reformulate the previous results with a half-form bundle.

\subsection{Definitions} 

Recall first some definitions. A {\em square root} $(B, \varphi)$ of a complex line bundle $A
\rightarrow \mathcal{N}$ over a manifold $\mathcal{N}$ is a complex line bundle $B \rightarrow \mathcal{N}$ with an isomorphism
$\varphi : B^{\otimes 2} \rightarrow A$. A {\em half-form bundle} of a complex manifold
is a square root of its canonical bundle. Since the group of isomorphism
classes of complex line bundles of a manifold $\mathcal{N}$ is isomorphic to $H^2(\mathcal{N})$,
the isomorphism being the Chern class, a sufficient condition for a complex
manifold to have a half-form bundle is that its second cohomology group is
trivial. This condition will be sufficient for our purposes. Before we discuss
the uniqueness, let us explain how derivatives and connections can be
transferred from a bundle to its square roots.  

Assume that $(B, \varphi)$ is a square root of $A$. Then any derivative $D_B$ acting
on sections of $B$ induces a derivative $D_A$ acting on sections of $A$ such
that the Leibniz rule is satisfied
$$D_A ( u \otimes v ) = D_B( u) \otimes v +
u \otimes D_B (v), \qquad \forall \, u,v \in \Ci ( B)$$
The converse is true as well: any derivative $D_A$
of $A$  determines a derivative $D_B$ of $B$ such that the above identity is satisfied. Similarly a covariant derivative $\nabla^B$ of $B$ induces a covariant
derivative $\nabla^A$ of $A$ such that $\nabla^A( u \otimes v ) = \nabla^B(u)
\otimes v + u \otimes \nabla^B (v)$, and the converse holds as well. 

Two square roots $(B, \varphi)$ and $(B', \varphi')$ of $A$ are isomorphic if
there exists a line bundle isomorphism $\psi : B \rightarrow B'$ such that
$\varphi' \circ \psi^{2} = \varphi$. The isomorphism classes of square roots of
the trivial line bundle $\C_{\mathcal{N}}$ of $\mathcal{N}$ are in bijection with $H^1(\mathcal{N}, \Z_2)$. Indeed,
each square root of $\C_{\mathcal{N}}$ has a natural flat structure with holonomy in
$\{ -1 , 1 \} \subset \op{U}(1)$, induced by the flat structure of $\C_{\mathcal{N}}$. We easily check this determines the square root up to
isomorphism. Furthermore, the tensor product of line bundles defines an action
of square roots of $\C_{\mathcal{N}}$ on the space of square roots of a
given line bundle $A$. This makes the set of isomorphism classes of square
root of $A$ a homogeneous space for the group $H^1(\mathcal{N}, \Z_2)$. 

\subsection{Propagation in terms of half-form bundle} 

When $M$ has a half-form bundle $\delta$, we can reformulate the previous
results by introducing a new line bundle $L_1$ such that $L' = L_1 \otimes
\delta$. The relevant structures of $L_1$ and $\delta$ have a different
nature: 
\begin{itemize}
\item $L_1$ has a natural connection, its Chern connection, 
\item the restriction of $\delta$ to a Lagrangian submanifold $\mathcal{N}$ of $M$ is a square
  root of $\op{det}( T^*\mathcal{N} ) \otimes \C$, through the isomorphism $K|_{\mathcal{N}} \simeq \op{det}( T^*\mathcal{N} ) \otimes \C$.  
\end{itemize}
For instance, in our propagation results, on the one hand, the tangent
map to the flow defines a map from $\det (T^* \Ga_0)$ to $\det( T^* \Ga_t)$,
which gives the map  
$\mathcal{E}_t: K|_{\Ga_0} \rightarrow K|_{\Ga_t}$. We then introduce the
square root of $\mathcal{E}_t$
$$ [\mathcal{E}_t(x)]^{\frac{1}{2}}: \delta_x \rightarrow \delta_{\phi_t(x)}, \qquad
x \in \Ga_0 $$
which is equal to the identity at $t=0$. On the other hand, we can define the
parallel transport $\mathcal{T}_t^{L_1}$ from the connection of $L_1$.  Then \eqref{eq:solution_transport}
writes equivalently
\begin{gather} \label{eq:solution_transport_metaplectique}   b (t,\phi_t(x)) =e^{\frac{1}{i}
\int_0^t H^{\s} (r, \phi_r (x)) \, dr  } \,   \mathcal{T}_t^{L_1} (x) \otimes [\mathcal{E}_t(x)]^{\frac{1}{2}} \, b(0,x)  
\end{gather}
The transport equation $(\nabla_Y + i \zeta )b =0$ has a similar formulation in
terms of the decomposition $L' = L_1 \otimes \delta$. Here it is convenient to
lift everything to $\Ga$. So we consider $\C_I \boxtimes L' \rightarrow
\Ga$ as the tensor product of $\C_I \boxtimes L_1 \rightarrow \Ga$ and
$\delta_\Ga := (\C_I \boxtimes \delta \rightarrow \Ga)$. Then the transport
equation is 
\begin{gather} \label{eq:transport_metaplectique}
\bigl( ( \nabla_Y^{L_1} \otimes \op{id} + \op{id} \otimes
\mathcal{L}_Y^\delta) + i H^{\s} \bigl) b =0  
\end{gather}
On the one hand, $\nabla^{L_1}$ is the Chern connection of
$L_1$ with derivative $\nabla_Y^{L_1}$ acting on $\Ci ( \Ga, \C_I \boxtimes
L_1)$. On the other hand, 
$\mathcal{L}_Y^\delta$  is the derivative
of $\Ci ( \Ga, \delta_{\Ga})$ induced by the Lie derivative $\mathcal{L}_Y$ of
$\Ga$ through the isomorphism 
$$\delta_\Ga^2 \simeq K_{\Gamma} \simeq
\det (T^* \Ga) \otimes \C$$
defined by \eqref{eq:iso_K_det_Gamma}. More precisely, $\mathcal{L}_Y^\delta$
is the unique derivative such that $\mathcal{L}_Y (s^2) = 2 s \otimes \mathcal{L}_Y^\delta s$ for any
section $s\in  \Ci ( \Ga, \delta_{\Ga})$. Then Formula \eqref{eq:transport_metaplectique} follows from the relation between $\zeta$ and $\theta$ and the fact that $\nabla^{L'}_Y b = ( \nabla_Y^{L_1} \otimes \op{id} + \op{id} \otimes \nabla_Y^\delta) b$ where $\nabla^\delta$ is the connection on $\delta$ induced by the one on $K$, which satisfies $\nabla^{\delta}_Y = \mathcal{L}_Y^\delta - \frac{i}{2} \theta$.

Interestingly, these formulations can be used even when $M$ has no half-form
bundle. To give a meaning to Equation
\eqref{eq:solution_transport_metaplectique}, we need a square root $\delta$ of the
restriction of $K$ to the trajectory $\phi_{[0,t]}(x)$ of $x$ on the interval
$[0,t]$. This trajectory being an arc or a circle, such a square root exists.
In the circle case, there are two square roots up to isomorphism, but it is easy to see that the
right-hand side of \eqref{eq:solution_transport_metaplectique} does not depend
on the choice. Similarly we can give a meaning to the transport equation
\eqref{eq:transport_metaplectique} even when $M$ has no half-form bundle. Indeed a differential
operator of $\Ga$ is determined by its restriction to the open sets of any covering of
$\Ga$. And we can always introduce a half-form bundle on the neighborhood of
each point of $M$.

\subsection{Norm estimates} 
The introduction of half-form bundles is also useful when we estimate the
norm of a Lagrangian state. For instance, consider a Lagrangian state $\Psi_k
(  t)$ as in \eqref{eq:def_lagrangian}. Then, by \cite[Theorem
3.2]{C2006}, 
\begin{gather}\label{eq:estimation_norm_psi_t} 
  \| \Psi_k (t)  \|_{\mathcal{H}_k} ^2 = \int_{\Ga_t} \Om_t + \bigo ( k^{-1})
  \end{gather} 
where $\Om_t$ is a density  on $\Ga_t$, which is given in terms of the
principal symbol $b_0 ( \cdot, t)$ of $\Psi_k ( t)$ as follows. We
assume that $L' = L_1 \otimes
\delta$ with $\delta$ a half-form bundle. Again we treat $\delta$ and $L_1$ in completely different ways. On the one hand, $L_1$ has a natural
metric so $L_1 \otimes \con{L}_1 \simeq \C$. On the other hand,  
$\delta|_{\Ga_t}$ being a square 
root of $\det(T^* \Ga_t) \otimes \C$, the identity $z \con{z} = |z^2|$ induces
an isomorphism between $ \delta|_{\Ga_t} \otimes \con{\delta}|_{\Ga_t} $ and
the bundle $|\bigwedge| T^* \Ga_t \otimes \C$ of densities. So we have an
isomorphism 
\begin{gather} \label{eq:norm_pointwise_symbol}
 L'|_{\Ga_t} \otimes \con{L'}_{\Ga_t} \simeq  |\textstyle{\bigwedge}| T^* \Ga_t \otimes
 \C.
\end{gather}
Then $\Om_t$ is the image of $b_0 (\cdot,t) \otimes \con{b}_0 ( \cdot , t)$ by
\eqref{eq:norm_pointwise_symbol}. When $M$ does not have a half-form bundle,
we can still define the isomorphism \eqref{eq:norm_pointwise_symbol} by
working locally and the global estimate \eqref{eq:estimation_norm_psi_t} still
holds. The normalization $\bigl(k/ 2\pi\bigr)^{\frac{n}{4}}$ in
the definition \eqref{eq:def_lagrangian} has been chosen to obtain this
formula.

Interestingly the isomorphism \eqref{eq:iso_K_det_Gamma} is also meaningful
for our norm estimates. Indeed, consider now $b_0$ as a section of $\C_I
\boxtimes L' \rightarrow \Ga$; repeating the previous considerations to $
(\C_I \boxtimes \delta)|_{\Ga}$ and $( \C_I \boxtimes L_1)|_{\Ga}$, we define
a density $\Om$ on $\Ga$ such that
\begin{gather} \label{eq:estimation_norm_psi} 
\int_I f(t)\; \| \Psi_k (t) \|_{\mathcal{H}_k} ^2 \; dt = \int_\Ga  (f\circ q) \; \Om + \bigo ( k^{-1}) , \qquad \forall \; f \in \Ci_0 (I). 
\end{gather}
where $q$ is the projection $\Ga \rightarrow I$. This follows from
(\ref{eq:estimation_norm_psi_t}), because $\Om_t$ is the restriction of
$\iota_{\partial_t} \Om$ to $\Ga_t$ and  a geometric version of Fubini's theorem tells us that
$\int_I f(t) \int_{\Ga_t} \Om_t = \int_\Ga (f\circ q ) \; \Om  $.

\section{The quantum propagator}
\label{sect:propagator}

In this section, we prove  Theorem \ref{th:intro1}. We will apply the previous considerations to $M \times \con{M}$, $L \boxtimes
\con{L}$ and $L' \boxtimes \con{L}'$ instead of $M$, $L$ and  $L'$.
The 
holomorphic sections of $(L\boxtimes \con{L})^k \otimes (L' \boxtimes
\con{L}')$ are the Schwartz kernels of the endomorphisms of $\Hilb_k$.

The symplectic structure of $\con{M}$ being the opposite of $\om$, the diagonal $\Delta_M$ is a
Lagrangian submanifold of $M \times \con{M}$. There is  a canonical flat section
$s : \Delta_M \rightarrow L \boxtimes \con{L}$ defined by $s(x,x) =  u \otimes \con{u}$ where $u \in L_x$ is any vector
of norm $1$. The Lagrangian states corresponding to $(\Delta_M , s)$
are the Toeplitz operators up to a factor $(\frac{k}{2 \pi})^{\frac{n}{2}}$.  More precisely, the Schwartz kernel of
$(\frac{k}{2 \pi})^{-\frac{n}{2}} T_k(f)$ is a Lagrangian state associated with $(\Delta_M, s)$
with principal symbol $f$, where we identify the restriction of $L' \boxtimes \con{L}'$ to the
diagonal with the trivial line bundle $\C_M = L' \otimes \con{L}'$ by using the
Hermitian metric of $L'$. This applies in particular to the identity of
$\mathcal{H}_k$, which is the Toeplitz operator $T_k(1)$ and is actually a reformulation of a theorem by Boutet de
Monvel-Sj\"ostrand \cite{BoSj,C2003op}. 

By Theorem \ref{th:Lagrangian_State_propagation}, the Schwartz kernel of the quantum propagator
$(U_{k,t})$ multiplied by $(\frac{k}{2 \pi})^{-\frac{n}{2}}$ is a Lagrangian state family, associated with the graph of
$\phi_t$ and its prequantum lift. Indeed, in the Schr\"odinger equation
\eqref{eq:schrod_equation}, we can interpret the product $T_{k,t} U_{k,t}$ as the
action of the Toeplitz operator $T_{k,t} \boxtimes \op{id}$ on $U_{k,t}$. Its
principal symbol is $H_t \boxtimes 1$, so its Hamiltonian flow is $\phi_t
\boxtimes \op{id}$. There is no difficulty to deduce Formula
\eqref{eq:leading_order_prop_quant} from Proposition \ref{prop:solution_transport_equation} 
except for the relation between  $\mathcal{E}_t$ and $\mathcal{D}_t$.

\begin{lm} \label{lem:toutestla}
  $\mathcal{E}_t
  (x,x) ( \op{id}_{K_x} ) = \mathcal{D}_t(x)$.
\end{lm}

Everything relies on the identification \eqref{eq:iso_Kt_volume} which in our
case is an isomorphism between $K_{\phi_t(x)} \otimes \con{K}_x$ and the space
of volume forms on the graph of $T_x \phi_t$. On the one hand, the elements of $K_{\phi_t(x)}
\otimes \con{K}_x$ will be viewed as  morphisms from $K_x$ to $
K_{\phi_t(x)}$. On the other hand, the graph
of $T_x \phi_t$ is naturally isomorphic with $T_xM$ through the map $\xi
\rightarrow (T_x \phi_t (\xi), \xi )$. So \eqref{eq:iso_Kt_volume} becomes an isomorphism 
\begin{gather} \label{eq:iso}
\op{Mor} ( K_x, K_{\phi_t(x)} ) \simeq \det (T^*_x M) \otimes \C
\end{gather}
Now the tangent map to the flow $\phi_t \boxtimes \op{id}$ sends the
graph of $T_x \phi_0$ to the graph of $T_x \phi_t$, and with our
identifications, it becomes the identity of $T_x M$. So the map $\mathcal{E}_t
(x,x)$ is the isomorphism
$$ \op{Mor} ( K_x, K_x) \simeq \op{Mor} ( K_x, K_{\phi_t(x)})  $$
otained by applying \eqref{eq:iso} with $t=0$ and then the inverse of
\eqref{eq:iso}. 

\begin{proof}[Proof of Lemma \ref{lem:toutestla}, technical part] First we
  claim that \eqref{eq:iso} sends a morphism $\psi: K_x \rightarrow K_{\phi_t(x)}$ to
$$ ((T_x \phi_t)^* \psi(\al) ) \wedge \con{\al} $$ where $\al \in K_x$ is any
vector with 
norm $1$. Indeed, $\psi$ is first identified with $\psi(\al) \otimes \con{\al}
\in K_{\phi_t(x)} \otimes \con{K}_x$. Then it is viewed as the $2n$-form of $T_xM
\oplus T_xM$ given by $p_1^* \psi(\al) \wedge p_2^* \con{\al}$ where $p_1$ and
$p_2$ are the projections $T_xM \oplus T_x M \rightarrow T_xM$ onto the first
and the second factor respectively. Then it is restricted to the graph of $T_x
\phi_t$ which is identified with $T_xM$ via the map $h(\xi ) = ( T_x \phi_t
(\xi), \xi)$, so we obtain
$$ h^{*} (p_1^* \psi(\al) \wedge p_2^* \con{\al}) = (( T_x \phi_t)^* \psi(\al)
) \wedge \con \al$$
because $p_1 \circ h = T_x \phi_t$ and $p_2 \circ h = \op{id}$. 

For $t=0$ and $\psi = \op{id}$, we have $(( T_x \phi_t)^* \psi(\al)
) \wedge \con \al = \al \wedge \con \al$.   
So we have to prove that for $\be =\mathcal{D}_t(x) (\al)$ 
\begin{gather} \label{eq:but}
((T_x \phi_t)^*  \beta ) \wedge \con{\al} = \al \wedge \con{\al}
\end{gather}
This is equivalent to $j^*(T_x \phi_t)^* \beta  = \al$
where $j$ is the injection $T_x^{1,0} M \rightarrow T_xM \otimes \C$. Since
$\beta \in K_{\phi_t(x)}$, we have $\pi^*\beta = \beta$ where $\pi$ is the
projection of $T_{\phi_t(x)}M \otimes \C$ onto the $(1,0)$-subspace with
kernel the $(0,1)$-subspace. So we have to show that $( \pi\circ(T_x \phi_t)
\circ j) ^* \beta = \al$. But $ \pi\circ(T_x \phi_t)
\circ j =  ( T_x \phi_t)^{1,0}$, so \eqref{eq:but} is equivalent to
$(( T_x \phi_t)^{1,0})^* \be = \al$. And this last equality is actually the definition
of $\be =  \mathcal{D}_t(x) (\al)$ .
\end{proof}

\begin{thm} \label{th:quantum_propagator}
  Let $(T_{k,t},\, t\in I)$ be a smooth family of 
  Toeplitz operators with real principal symbol $H_t$ and subprincipal symbol
  $H^{\s}_{t}$. Then the Schwartz kernel of the quantum propagator of
  $(T_{k,t})$ multiplied by $(\frac{k}{2 \pi})^{-\frac{n}{2}}$ is a Lagrangian state family associated with $(\Ga, s , \si)$
  given by $\Ga = \{ ( t,
  \phi_t(x),x) /\; t \in I, \; x \in M \}$ and
\[	\begin{cases} s( t, \phi_t(x), x)  = \phi_t^L (x): L_x \rightarrow
   L_{\phi_t(x)}, \\[2mm]
   \si (t, \phi_t(x), x )  = \bigl[ \rho_t(x) \bigr]^{\frac{1}{2}} e^{\frac{1}{i} \int_0^t H^{\s}_{r} (\phi_r(x)) \;
    dr } \mathcal{T}_t^{L'} (x): L'_x \rightarrow L'_{\phi_t(x)}, 
\end{cases} \]
where $(\phi_t)$ is the Hamiltonian flow of $H_t$, $\phi_t^L$ its prequantum
lift, $\mathcal{T}^{L'}_t$ its parallel transport lift to $L'$ and  $
\mathcal{D}_t(x) = \rho_t(x) \mathcal{T}_t^K(x)$ with $\mathcal{D}_t(x) = K(
T_x \phi_t): K_x
\rightarrow K_{\phi_t(x)}$. 
\end{thm}
As explained in the introduction, it is very natural to express the symbol by
using a half-form bundle:
$$ \si (t, \phi_t(x), x ) = e^{\frac{1}{i} \int_0^{t} H^{\s}_{r} ( \phi_r(x)) \; dr}  \mathcal{T}_{t}^{L_1}
(x)  \otimes  \bigl[ \mathcal{D}_t (x) \bigr]^{\frac{1}{2}} 
$$
where $L' = L_1 \otimes \delta$ and $\bigl[ \mathcal{D}_t (x) \bigr]^{\frac{1}{2}}: \delta_x \rightarrow
\delta_{\phi_t(x)}$ is the continuous square root of $\mathcal{D}_t(x)$ equal to $1$ at $t=0$.  

\begin{rmk}\label{rmk:prep_trace} In our next paper on trace formulas, we will use the following
  expression for $\rho_t(x)$. Denote by $\ga : \R \rightarrow M$, $ t
  \mapsto \phi_t(x)$ the trajectory of
  $x$. Choose a unitary frame $s_K$ of $\ga^*K$ and write $\nabla  s_K
  = \frac{1}{i} f_K dt \otimes s_K$. Then
\begin{gather} \label{eq:rho_tx-=-c_t}
  \rho_t(x) = c_t e^{-i \int_0^t f_K(r) \, dr}  \quad \text{ where } \quad  c_t  \Om_t
  ( \xi_t  u) = \Om_0 ( \xi_0 u)  .
\end{gather}
Here $u$  is any generator of $\wedge^{\op{top}}
 T_xM$, $\xi_t$ is the linear map  $T_x M \rightarrow T_{\ga(t)}M \oplus T_x M$ sending $X$
  into $(T_x \phi_t(X), X)$, and $\Om_t$ is the $2n$-form of $T_{\ga(t)}M \oplus
  T_x M$ equal to $\Om_t
  = p_1^* s_K(t) \wedge p_2^* \con{s}_K(0)$, $p_1$ and $p_2$ being the projection on
  $T_{\ga(t)}M$ and $T_xM$ respectively. 

  The proof of \eqref{eq:rho_tx-=-c_t} is that on the one hand $\mathcal{E}_t
  (x,x) \Om_0 = c_t \, \Om_t$ and on the other hand  $\mathcal{T}^{K
    \boxtimes \bar{K}}_t (x,x) \Om_0 = e^{i \int_0^t f_K(r)dr} \Om_t$. \qed
\end{rmk}

\section{Fourier Transform of Lagrangian state families}
\label{sect:imm_lag}

In this section, we investigate how the (inverse) semiclassical Fourier transform acts on the Lagrangian state families introduced in Section \ref{sec:famil-lagr-stat}. It turns out that the outcomes are states which are associated with Lagrangians that are only immersed; hence we need to generalize the usual definition of Lagrangian states recalled at the beginning of Section \ref{sect:lag_fam}.

\subsection{Symplectic preliminaries} 
Consider the same data $\Ga \subset I \times M$ and $s\in \Ci ( \Ga, \C_I
\boxtimes L)$ as in Section \ref{sec:famil-lagr-stat}. So we assume that $\Ga
\rightarrow I$, $(t,x) \mapsto t$ is a proper submersion and that for any $t
\in I$, $\Ga_t$ is a Lagrangian submanifold of $M$ and the restriction of $s$
to $\Ga_t$ is flat and unitary. Recall that $\nabla s = i \tau dt \otimes s$ for a function
$\tau \in \Ci ( \Ga)$, cf \eqref{eq:def_tau}.
For any $E$ in $\R$, introduce
\begin{gather} \label{eq:def_gae}
\Ga^E := \{ (t,x) \in \Ga \ | \ \tau (t,x) + E =
0 \} .
\end{gather}
\begin{prop} \label{prop:symplectic_Gamma_E}
  Let $E$ be a regular value of $- \tau$. Then
  \begin{enumerate}
\item $\Ga^E$ is a submanifold of $\Ga$ and $j_E : \Ga^E \rightarrow M$, $(t,x)
  \mapsto x$ is a Lagrangian immersion,
  \item  $j: \Ga \rightarrow \R \times M$, $(t,x)
    \mapsto (\tau (t,x) , x)$ is an immersion at any $(t_0,x_0) \in \Ga^E$, 
\item the section $s^E$ of $(j^E)^*L$ given by $s^E(t,x) = e^{it E }
  s(t,x)$ is flat. 
  \end{enumerate}
\end{prop}
\begin{proof}
For any tangent vectors $Y_1 = ( a_1, \xi_1)$, $Y_2 = (a_2, \xi_2)$ in $T_{(t_0,x_0)}
\Ga \subset \R \oplus T_{x_0} M$, we have
\begin{gather} \label{eq:relation_utile}
  \om (\xi_1, \xi_2) = a_1 d\tau (Y_2) - a_2 d \tau ( Y_1)
\end{gather}
To prove this, we extend $Y_1$ and $Y_2$ to vector fields of $\Ga$ on a neighborhood of $(t_0,x_0)$
so that $[Y_1, Y_2]
=0$. Then, the curvature of $\nabla$ being $\frac{1}{i} \om$,  we have that
$$ [\nabla_{Y_1} , \nabla_{Y_2} ] = \tfrac{1}{i} \om ( \xi_1, \xi_2).$$
Furthermore, since $dt(Y_j)  = a_j$, we have that $\nabla_{Y_j} s = i a_j \tau \otimes s$ so $\nabla_{Y_1}
\nabla_{Y_2} s = i ( (Y_1 .a_2) \tau +a_2 (Y_1. \tau)) s - a_1 a_2 \tau^2 s$.
Using that $ Y_1.a_2 - Y_2.a_1 = [Y_1, Y_2] .t = 0$, it comes that
$$ [\nabla_{Y_1} ,
\nabla_{Y_2} ] s = i ( a_2 (Y_1.\tau) - a_1 ( Y_2 .\tau)) s.$$
Comparing with the previous expression for the curvature, we obtain \eqref{eq:relation_utile}.

We prove the second assertion. Assume $Y_1 =  (a_1,\xi_1)$ is in the kernel of
the tangent map of $j: \Ga \rightarrow \R \times M$, that is $d\tau (Y_1) = 0$ and $\xi_1=0$. Then
\eqref{eq:relation_utile} writes $0 = a_1 d \tau (Y_2)$. If $a_1 \neq 0$, this
implies that $d\tau (Y_2) =0$ for any $Y_2 \in T_{(t_0, x_0)} \Ga$, which contradicts the assumption that $-E$ is a regular value of $\tau$. 

This implies that $j_E$ is an immersion. It is Lagrangian by
\eqref{eq:relation_utile} again because if $Y_1$, $Y_2$ are tangent to $\Ga^E$, then
$d\tau ( Y_1) = d\tau (Y_2) =0$, so $\om (\xi_1,\xi_2) =0$. Finally,
$ \nabla ( e^{it E  } s) = (iE dt + i \tau dt )\otimes  e^{i tE } s = 0$ on
$\Ga^E$. 
\end{proof}

\subsection{Immersed Lagrangian states} 
\label{sec:immers-lagr-stat}

We will adapt the definition of Lagrangian states for immersed manifolds.
Suppose we have a Lagrangian immersion $j:\mathcal{N} \rightarrow M$, a flat unitary
section $s$ of $j^*L$ and a formal series $\sum \hbar^\ell b_\ell$ with
coefficients $b_{\ell} \in \Ci ( j^* L')$.

First, for any $y \in \mathcal{N}$, we will define a
germ of Lagrangian state at $j(y)$, uniquely defined up to $\bigo (
k^{-\infty})$ as follows. Let us assume temporarily  that there exists an open set $V$ in $M$ such that $j:\mathcal{N} \rightarrow V$
is a proper embedding, so that $j(\mathcal{N})$ is a closed submanifold of $V$. Then we
can introduce sections $F : V \rightarrow L$
and $a_{\ell}: V \rightarrow L'$ such that $\con \partial F$ and $\con \partial
a_{\ell}$ vanish to infinite order along $j(\mathcal{N})$, $j^* F = s$ and $j^* a_\ell =
b_{\ell}$ and $|F|<1$ on $V \setminus j(\mathcal{N})$. These sections are not unique but
if $(F', a'_{\ell}, \, \ell \in \N)$ satisfy the same condition, then
for any $N$,
\begin{gather} \label{eq:unique_DAS}
F^k \sum_{\ell =0 }^N k^{-\ell} a_{\ell} = (F')^k
\sum_{\ell=0}^N k^{-\ell} a'_{\ell} + \bigo ( k^{-N-1})
\end{gather}
the $\bigo $ being uniform on any compact set of $V$. This follows on one hand
from the
fact that $|F|$ and $|F'|$ are $<1$ on $V \setminus j(\mathcal{N})$, so that both sides
of \eqref{eq:unique_DAS} are in $\bigo ( k^{-N-1})$ uniformly on any compact set of $V \setminus j(\mathcal{N})$.
On the other hand, the sections $F$, $F'$ and $a_{\ell}$, $a_{\ell}'$ have the
same Taylor expansions along $j(\mathcal{N})$ which implies \eqref{eq:unique_DAS} on a
neighborhood of $j(\mathcal{N})$,  (see \cite[Section 2.2]{C2003} for details).

Back to a general immersion $\mathcal{N} \rightarrow M$, for any $y \in \mathcal{N}$, by the local normal form for immersions, there exists open neighborhoods $U$ and $V$ of
$y$ and $j(y)$ respectively such that $j(U) \subset V$ and $j$ restricts to a closed embedding from
$U$ into $V$. Then we can introduce the sections $F$ and $a_{\ell}$, $\ell \in
\N$ as above on $V$, which extend the restrictions of $s$ and $b_{\ell}$ to
$U$. This defines the expansion
\begin{gather} \label{eq:expansion_local}
\Psi_{N,k} := F^k \sum_{\ell =0 }^N k^{-\ell}
a_{\ell} 
\end{gather}
on $V$. If we have another set of data $(U',V', F' , a'_{\ell})$, we obtain another
sequence $\Psi_{N,k}' := (F')^k \sum_{\ell =0 }^N k^{-\ell}
a'_{\ell}$ on $V'$.
\begin{lm}  For any $N$, $\Psi_{N,k} = \Psi'_{N,k} + \bigo ( k^{-N-1})$ on a
neighborhood of $j(y)$. 
\end{lm}
So we have a well defined germ of Lagrangian states at $j(y)$. 
\begin{proof} Choose open sets $W$ and $W'$ of $V$ and $V'$ respectively
such that
$j(U \cap U') =  j(U) \cap W = j(U') \cap W'$. Set $U'' = U \cap U'$ and $ V''
=  W \cap W'$. Then $j$ restricts to an embedding from $U''$ into $V''$ and
$j(U'') = j (U') \cap V''$. So the restriction of $F$, $a_{\ell}$ to $V''$
gives us a new set of data $(U'', V'', F|_{V''}, a_\ell|_{V''})$. The fact
that $j(U'') = j (U') \cap V''$ is used to see that $|F|<1$ on $V'' \setminus
j(U'')$. 
Similarly,
we can restricts $F'$, $a'_{\ell}$ to $V''$ and get $(U'', V'', F'|_{V''},
a_{\ell}'|_{V''})$. The final result follows from our initial remark \eqref{eq:unique_DAS}. 
\end{proof}

Now assume that there exists a compact subset $K$ of $\mathcal{N}$ such that for each
$\ell$, $a_{\ell}$ is supported in $K$. Our goal is to construct a Lagrangian
state which, on a neighborhood of each $x\in M$, is equal to the sum of
the local Lagrangian states defined previously for each $y \in j^{-1} (x) \cap
K$. An essential observation is that $j^{-1} (x)$ is discrete in $\mathcal{N}$, so $j^{-1} (x) \cap K$ is finite since $K$ is compact.  

\begin{lm} 
  There exists a family $(\Psi_k \in \Hilb_k)$ such that for any $x \in M$ and
  any $N$:
  \begin{enumerate}
    \item if $x \notin j(K)$,  $|\Psi_k|  = \bigo ( k^{-N})$ on a
      neighborhood of $x$,
     \item if $j^{-1} (x) \cap K = \{ y_i , \, i \in I\}$, then $\Psi_k =
       \sum_i \Psi^i_{N,k} + \bigo ( k^{-N})$ on a
       neighborhood of $x$, where each $ \Psi^i_{N,k}$ is defined as in
       \eqref{eq:expansion_local} with $y = y_i$. 
     \end{enumerate}
   \end{lm}
We will call $(\Psi_k)$ a Lagrangian state associated with the Lagrangian immersion $j : \mathcal{N}
\rightarrow M$, the flat unitary section $s$ of $j^*L$ and the formal series
$\sum \hbar^\ell b_\ell$ with coefficients in $\Ci ( \mathcal{N}, j^* L')$.  $(\Psi_k)$
is unique up to $\bigo ( k^{-\infty})$. But unlike the case of a Lagrangian
submanifold, we can not recover the symbol $\sum \hbar^\ell b_{\ell}$ from
the state by taking the restriction to $j(\mathcal{N})$ because of the possible multiple
points.  
   
   \begin{proof} 
Consider an open set $V$ of $M$ and a finite family $(U_i)_{i \in I}$ of
disjoint open
sets of $\mathcal{N}$ such that for any $i \in I$, $j$ restricts to a proper embedding from $U_i$ into $V$
and $K \cap j^{-1} (V) \subset \bigcup U_i$. Then introduce sections $F_i$,
$a_{i, \ell}$ on $V$ as above associated with each submanifold $j(U_i)$.
Consider the sum
\begin{gather} \label{eq:lag_state_im} 
\Bigl(\frac{k}{2\pi} \Bigr)^{\frac{n}{4}} \sum_{i \in I} F_i^k(x) \sum _{\ell=0 }^N
k^{-\ell} a_{\ell}(x) , \qquad x \in V .
\end{gather}
Then for any $x \in V$, each $y \in j^{-1} (x) \cap K$ belongs to one of the
$U_i$, so on a neighborhood of $x$, \eqref{eq:lag_state_im} is equal to the
sum of the Lagrangian state germs associated with the $y \in j^{-1}(x) \cap K$.
So by the previous discussion, the state defined by \eqref{eq:lag_state_im} on
a neighborhood of $x$ does not depend, up to $\bigo ( k^{-\infty})$, on the choice
of $V$ and of $(U_i, F_i, a_{i, \ell})$ for $i \in I$ and $\ell \in \N$. It is not
difficult to prove that any point $x$ of $M$ has an open neighborhood $V$
admitting a family $(U_i)$ as above. Indeed, if we set $I =
j^{-1}(x) \cap K$, then for any $y \in I$, there exists a pair $(U_y \ni y ,V_y)$ such that $j$ restricts to a proper
embedding from $U_y$ into $V_y$. Then we choose for $V$ a sufficiently small
neighborhood of $x$ in $\cap_y V_y$ such that $j^{-1} (V) \cap K \subset \cup_y
U_y$ and we restrict the $U_y$ accordingly.  
So with a partition of unity, we can construct global states $\Psi_k \in \Ci (
M , L^k \otimes L'), k \in \N$, such that for any data $(V, U_i, i \in I)$ as
above, $\Psi_k$ is equal to \eqref{eq:lag_state_im} on $V$ up
to a $\bigo ( k^{-\infty})$, uniform on any compact subset of $V$. Since $\con
\partial \Psi_k$ is in $\bigo ( k^{-\infty})$, we can replace $\Psi_k$ by its
projection onto $\Hilb_k$, which only modifies it by a $\bigo ( k^{-\infty})$ by
Kodaira-H\"ormander estimates \cite{Kod,Hor_dbar}. 
\end{proof}

\subsection{Fourier transform}
Introduce the $\hbar$-Fourier transform and its inverse with parameter $k = \hbar^{-1}$
$$ \mathcal{F}_k(f) (t) = \Bigl(\frac{k}{2 \pi} \Bigr)^{\frac{1}{2}} \int_{\R} e^{-ik tE} f(E) \, dE, \quad
 \mathcal{F}_k^{-1}(g) (E) = \Bigl(\frac{k}{2 \pi} \Bigr)^{\frac{1}{2}}
 \int_{\R} e^{iktE} g (t) \, dt .$$
 We are now ready to state the main result of this section. 
\begin{thm} \label{th:Lagrangian_Fourier}
Let $(\Psi_k \in \Ci ( I , \Hilb_k), \; k \in \N)$ be a Lagrangian state
family associated with $(\Ga, s)$  and such
that the $\Psi_k$ are supported in a compact set of $I$ independent of $k$.
Let $-E$ be a regular value of $\tau$ and $\Ga^E := \tau^{-1}(-E)$, using the notation introduced before and in Equation \eqref{eq:def_gae}.

Then $\mathcal{F}^{-1}_k (\Psi_k)(E)$ is
a Lagrangian state associated with the Lagrangian immersion $j_E: \Ga^{E}
\rightarrow M$, $(t,x) \mapsto x$, the section $s^E \in \Ci ( \Ga^{E}, j_E^*
L)$ given by $ s^E ( t,x) = e^{itE} s (t,x)$  and the principal symbol 
$$ \si^{E} (t,x) = B(t,x)^{-\frac{1}{2}} \si ( t,x), \qquad (t,x) \in \Ga^E $$
where $ \si $ is the principal symbol of $(\Psi_k)$ and $B (t,x)$ is such that
$d \tau \wedge \al = i B (t,x) \, dt \wedge \al$ on
$T_{(t,x)} \Ga$ for
any $\al \in K_x$, the square root $B(t,x)^{1/2}$ having a non negative real part.
\end{thm}

We already explained that for a non-zero $\al \in K_x$, $dt \wedge \al$ is nonzero on $T_{(t,x)} \Ga$. By the second assertion of Proposition
\ref{prop:symplectic_Gamma_E},  the same argument shows that $d \tau \wedge \al$ is
non zero on $T_{(t,x)} \Ga$ when $(t,x) \in \Ga^E$. This proves that $B(t,x)$ is uniquely determined
and nonzero as well.  

\begin{proof}
Introduce a local unitary frame $u$ of $L$ and write $F (t,x) = e^{ i f(t, x)
} u (x)$, where $F$ is the section appearing in the definition
\eqref{eq:def_lagrangian} of $\Psi_k$.  Then
\begin{gather} \label{eq:int}
\mathcal{F}^{-1}_k(\Psi_k)(E)(x) = \Bigl( \frac{k}{2 \pi}
\Bigr)^{\frac{n}{4} + \frac{1}{2} }  u^k (x) \int_{\R} e^{ik \phi (t,x )} a(t,x,k) \,
  dt   
\end{gather}
with $\phi(t,x) = tE + f(t,x)$. The imaginary part of $\phi$ is nonnegative.
It vanishes when $|F(t,x)| =1$, that is when $(t,x) \in \Ga$. We have
$$ \phi'_t (t,x) = E + f'_t (t,x), \qquad \phi''_{tt} (t,x) = f''_{tt}
(t,x). $$
Here $g \mapsto g'_t$ means differentiation with respect to $t$.

We claim that the function $f'_t $ is an extension of $\tau$ such that $\con \partial f'_t$
vanishes to infinite order along $\Ga$. Indeed, by taking the restriction of $
\nabla F$ to $\Ga$, we obtain that $f'_t = \tau$ on $\Ga$ (see also the argument in the proof of Proposition \ref{prop:derive_etat_lag}). Then since $\con
\partial F $ vanishes to infinite order along $\Ga$, the same holds for $\con \partial f
+ \frac{ \con \partial u}{u} $, and by taking the derivative with respect to $t$, the same
holds for $\con \partial f'_t$. 

So for $(t,x) \in \Ga$, $\phi'_t(t,x) = 0$ if and only if $(t, x) \in \Ga^E$.
So by Lemma 7.7.1 in \cite{Ho}, if $(t_0 , x_0) \notin \Ga^E$, then the
integral in \eqref{eq:int} restricted to a neighborhood of $t_0$ is in $\bigo (
k^{-\infty})$ on a neighborhood of $x_0$. So to estimate \eqref{eq:int} on a
neighborhood of a point $x_0$, it suffices to integrate on a neighborhood of
$j^{-1}_E (x_0)$. Let $V$, $U$ be neighborhoods of $x_0$ and $t_0 \in j^{-1}_E(
x_0)$ respectively such that $W = \Ga^E \cap
( U \times V)$ is a graph $\{ (t(x), x) , \, x \in j_E (W) \cap V \}$. 

Since $f'_t = \tau$ and $\con{\partial} f'_t=0$ on $\Ga$, we have $f''_{tt} dt
+ \partial f'_t = d \tau$ on $\Ga$. Mutiplying by $\al \in \Om^{n,0} (M)$, we
get $ f''_{tt} dt \wedge \al = d\tau \wedge \al$ on $\Ga$. As explained before
the proof, $d \tau \wedge \al$ does not vanish on $T_{t,x} \Ga$, so $f''_{tt} $
does not vanish on
$\Ga^E$ and we can apply the stationary phase lemma for a 
complex valued phase, see \cite[Theorem 2.3]{MeSj} or \cite[Theorem 7.7.12]{Ho}.
This theorem implies that on a neighborhood of $j_E(W)$ 
\begin{gather} \label{eq:local_int}
\Bigl( \frac{k}{2\pi} \Bigr)^{\frac{1}{2}} \int_{U} e^{ik \phi (t,x) }
a(t,x
,k) \, dt = e^{ik \phi_E (x)}  \sum_{\ell =0 }^N k^{-\ell} a_{E, \ell} (x) +
\bigo ( k^{-N-1}) 
\end{gather}
for any $N$, where
$$\phi_E (x) = \phi ( T(x), x), \qquad a_{E,0}(x) = (-i \phi_{tt}''(t,x))^{-\frac{1}{2}}
a_0 ( T(x) ,x),$$
the square root having a non negative real part. 
Here $T: U \rightarrow \C$ is an extension
of $x \rightarrow t(x)$, that is $(T(x),x) \in \Ga^ E$ when $ x\in j_E(W)$. The extension is chosen so that   $ \phi_t' ( T(x), x) =0$, where
$\phi$ itself has been extended almost
  analytically to a neighborhood
  of $\R \times M$  in $\C \times M$. We claim that $F (x) = e^{i \phi_E(x) }
  u(x)$ is adapted to $ (j_E(W), s^E|_W )$. First if $x \in j_E(W)$, then
  $$ e^{i \phi_E(x) } u (x) = e ^{it(x)E +i f(t(x), x) } u(x)=  e^{i t(x) E}
  s(t(x), x) = s^E ( t(x), x)$$
 It remains to show that $\con \partial F $  vanishes to infinite order along
 $j_E(W)$. Assume first that we can choose the section $F$ to be holomorphic, so
 that $\phi_t'$ depends holomorphically on $x$. If furthermore we can extend
 $\phi'_t$ so that it depends holomorphically on $t$, then by the holomorphic
 version of the implicit function theorem $T$ is holomorphic and $F$ is
 holomorphic as well. In general, we only know that $\con \partial F$ vanishes
 to infinite order along $\Ga$ and by adapting the previous argument, we
 conclude that $\con \partial F$ vanishes to infinite order along $j_E( W)$.
 
 With a similar proof, we can also show that the same holds for the coefficients: $\con
 \partial a_{E, \ell} \equiv 0$ along $j_E(W)$ to infinite order.
However, it is actually easier to use the following fact: $F^k \sum
 k^{-\ell} b_{\ell} = \bigo ( k^{-\infty})$ if and only if all the coefficients
 $b_{\ell}$ vanish to infinite order along $j_E(W)$. And here we know that 
 $\con \partial ( F^k
 \sum k^{-\ell} a_{E, \ell} ) = \bigo ( k^{-\infty})$ by differentiating
 \eqref{eq:local_int} under the integral sign.
\end{proof}

\section{Spectral projector}
\label{sect:proj}

Consider a self-adjoint operator $\hat{H}$ acting on a finite dimensional
Hilbert space $\mathcal{E}$. Here it is important that $\hat{H}$ is
time-independent. Introduce a smooth function $f: \R \rightarrow \C$ with smooth
compactly supported Fourier transform $\hat f$. We will work with the unitary
Fourier transform, so
$$ \hat{f} (t) = \frac{1}{\sqrt{2 \pi}} \int_{\R} e^{-itE} f(E) \, dE, \qquad
f (E) = \frac{1}{\sqrt{2 \pi}} \int_{\R} e^{itE} \hat f (t) \, dt$$
The second formula directly gives 
$$f (\hbar^{-1} (E -\la) ) = \frac{1}{\sqrt{2 \pi}} \int_{\R} e^{\frac{itE}{\hbar} }
\, e^{-\frac{it \lambda}{\hbar}}\, \hat f (t) \, dt $$
with $\hbar $ and $\la$ two real parameters.  Doing a spectral decomposition $\hat{H} = \sum \la
\Pi_{\la}$ where the $\la$ and $\Pi_{\la}$ are the eigenvalues and spectral
projectors, and introducing the quantum propagator $U_t = \exp ( - \frac{i
t \hat H}{ \hbar} ) = \sum e^{-\frac{it \lambda}{\hbar}} \Pi_{\la}$, we obtain
\begin{gather*} 
f (\hbar^{-1} (E - \hat{H} ) ) = \frac{1}{\sqrt{2 \pi}} \int_{\R}
e^{\frac{itE}{\hbar} }\,  U_t  \, \hat f  (t) \, dt 
\end{gather*}
We can apply this to our Toeplitz operator $(T_{k} : \Hilb_k \rightarrow
\Hilb_k)$ with quantum propagator $U_{k,t}$, which gives
\begin{gather} \label{eq:spectral_projector}
f\bigl( k ( E - T_{k})\bigr) = k^{-\frac{1}{2}} \mathcal{F}_k^{-1} ( \hat{f} (t)
U_{k,t} ) (E)
\end{gather}
If $E$ is a regular value of the principal symbol $H$ of $T_k$, we
deduce by Theorem \ref{th:Lagrangian_Fourier} that the Schwartz kernel of this
operator is a Lagrangian state. This will be done in Section
\ref{sec:smooth-spectr-proj} and will prove Theorem \ref{th:intro2} of the
introduction. Before that, we will consider the simpler case of a state $(\Psi_{k} \in \Hilb_k)$:
\begin{gather} \label{eq:state_decomposition}
f\bigl( (k ( E - T_{k}) \bigr) \Psi_{k}  = k^{-\frac{1}{2}} \mathcal{F}_k^{-1} ( \hat{f}
(t) \Psi_{k,t} ) (E)
\end{gather}
where $\Psi_{k,t}$ is the solution of $(\frac{1}{ik} \partial_t + T_k )
\Psi_{k,t} =0$ with initial condition $\Psi_{k}$.

\subsection{Lagrangian state spectral decomposition} \label{sec:lagr-state-spectr}

Let $\Ga_0$ be Lagrangian submanifold of $M$ and let $H$ be an autonomous
Hamiltonian with flow $\phi_t$. Set $\Ga_t = \phi_t ( \Ga_0)$ and $\Ga =  \{ ( t, x) \ | \ x\in \Ga_t \}$. Let $E\in \R$
be a regular value of the restriction of $H$ to $\Ga_0$ and define
$$ \Ga^E =\{ ( t,x) \ | \  x \in \Ga_{t} \cap H^{-1} (E) \}.$$ 
This submanifold is the same as the submanifold $\Ga^E$ defined in \eqref{eq:def_gae}
from a (local) flat section  $s_0$ of $L\rightarrow \Ga_0$, because the
corresponding function $\tau$ is the restriction of $-H$ to $\Ga$, see
\eqref{eq:function_tau}. Furthermore
our assumption on $E$ is equivalent to the fact that $-E$ is
regular value of $\tau$. 

The computation of the symbol of $f(k (E - T_k)) \Psi_k$ in terms of the
symbol of $\Psi_{k,0}$ will amount to transforming a volume element of $\Ga_0$
into a volume element of $\Ga^E$. Let us explain this.  
We denote by $X$ the Hamiltonian vector field of $H$ and by $X_x^{\perp_{\om}}$ the
symplectic orthogonal of $\R X_x$ in $T_xM$.  For any $(t, x)  \in \Ga^E$,
the Lagrangian space
\begin{gather} \label{eq:def_ltx}
\mathfrak{L}_{(t,x)}  := \R X_x \oplus (T_x
\Ga_t \cap X_x^{\perp_{\om}})
\end{gather}
is the image of $T_{(t,x)} j_E$ with $j_E: \Ga^E \rightarrow M$ the projection $(t,x) \mapsto
x$. Observe that $X_x$ does not
belong to $T_x \Ga_t$ because $E$ is a regular value of $H|_{\Ga_0}$, so it a
regular value of $H|_{\Ga_t}$ as well. 

Assume now $t=0$ and $(0,x) \in \Ga^E$. Choose $\eta \in T_x \Ga_0$ such that $\om (X_x, \eta) = 1$. Then we have
$$ T_x \Ga_0 = \R \eta \oplus (T_x
\Ga_0 \cap X_x^{\perp_{\om}}), \qquad  \mathfrak{L}_{(0,x)} = \R X_x \oplus (T_x
\Ga_0 \cap X_x^{\perp_{\om}}).$$
Starting from $v \in \det
T_x\Ga_0$, we write $v = \eta \wedge w $ with $w \in
\det( (T_x \Ga_0)\cap X_x^{\perp_{\om}})$ and we set $v(0,x) := X_x \wedge w
\in \det \mathfrak{L}_{(0,x)}$. This definition makes sense because $\eta$ is unique
modulo $(T_x
\Ga_0 \cap X_x^{\perp_{\om}})$ so that $w$ is unique. 
More generally, if $t$ is any real and $(0,x) \in \Ga^E$, we set $v(t,x) := (T_x \phi_t)_* v(0,x) \in \det \mathfrak{L}_{t,\phi_t(x)}$, viewing $T_x
\phi_t$ as a map from $\mathfrak{L}_{(0,x)}$ to $\mathfrak{L}_{(t, \phi_t(x))}$. 
Equivalently 
\begin{gather} \label{eq:def_vtx}
v(t,x) = X_{\phi_t(x)}
\wedge (T_x \phi_t)_* w .
\end{gather}

We now define a map $\mathcal{E}_t'(x) : K_x \rightarrow K_{\phi_t(x)}$ for
any $x \in \Ga_0 \cap H^{-1} (E)$ by
\begin{gather} \label{eq:def_eprime}
(\mathcal{E}_t'(x) \al) (v(t,x) ) = -i \al (v) , \qquad \forall v \in \det
T_x \Ga_0.
\end{gather}
We define the function $C'_t$ by the equality $\mathcal{E}_t'(x) = C'_t(x)
\mathcal{T}^K_t (x)$. 

\begin{prop} \label{prop:lagr-state-spectr}
  Let $(\Psi_k)$ be a Lagrangian state of $M$ associated with
  $(\Ga_0, s_0)$ with symbol $\si_0 \in \Ci ( \Ga_0, L')$, $(T_k)$ a
  self-adjoint Toeplitz operator with principal and subprincipal symbol $H$,
  $H^{\s}$, and  $f\in \Ci
  (\R)$ having a smooth compactly supported Fourier transform.

 If $E$ is a regular value of $H|_{\Ga_0}$,  then $\Psi'_k = k^{\frac{1}{2}} f\bigl( (k (
  E - T_{k}) \bigr) \Psi_{k}$ is a Lagrangian state associated with the
  Lagrangian immersion $j_E : \Ga^E \rightarrow M$, the flat section $s^E $ of
  $j_E^* L$ given by $s^E (t,\phi_t(x)) = \mathcal{T}^L_t(x) s_0(x)$ and the symbol
  $\si^E \in \Ci ( j_E^*L')$ defined as
  $$ \si^E (t,\phi_t(x)) = \hat{f} (t) \, C'_t(x)^{\frac{1}{2}} \, e^{\frac{1}{i}
\int_0^t H^{\s} ( \phi_r (x)) \, dr  } \,   \mathcal{T}_t^{L'} (x) \, \si_0 (x)  . $$
where the square root is chosen so as to be continuous and to have a positive
real part at $t=0$.  
\end{prop}

\begin{proof} 
The solution of the Schr\"odinger equation with initial condition
$\Psi_k$ is described as a Lagrangian state associated with $( \Ga, s(t,x) =
\phi_t^L(x) s_0(x))$ in Theorem \ref{th:Lagrangian_State_propagation}. Then
$\Psi'_k$ is the $k$-Fourier transform of this solution \eqref{eq:state_decomposition},
so by Theorem \ref{th:Lagrangian_Fourier}, it is a Lagrangian state associated with the immersion $j_E: \Ga^E
\rightarrow M$ and the section $s^E(t,x) = e^{itE} s(t,x) = \mathcal{T}_t^L (x)
s_0(x)$ because for an autonomous Hamiltonian, $\phi^L_t = e^{-it
  H}\mathcal{T}_t^L$, see \eqref{eq:prequantum_lift}. 

It remains to check the formula for the principal symbol. By Proposition
\ref{prop:solution_transport_equation}, we have to prove that $C'_t(x) =
\frac{C_t(x)}{B(t,\phi_t(x))}$, that is $\mathcal{E}_t'(x)=\frac{
 \mathcal{E}_t(x)}{B(t,\phi_t(x))}$, with $B$ the function of Theorem
\ref{th:Lagrangian_Fourier}. Comparing the definitions
\eqref{eq:def_E_transport} and \eqref{eq:def_eprime} of $\mathcal{E}_t(x)$
and $\mathcal{E}'_t(x)$, we have to show that for any $\beta \in
K_{\phi_t(x)}$ and $v \in \det(T_x \Gamma_0)$,
\begin{gather} \label{eq:buut}
B(t,\phi_t(x)) = \frac{i \beta (v(t,x))}{ \be( (T_x \phi_t)_* v)}
\end{gather}
where $v(t,x)$ is as in Equation \eqref{eq:def_vtx}. Let us first explain the proof at $t=0$. Recall that $v(0,x)  =  X_x \wedge w$ and
$v =  \eta \wedge w$. Now  
$B$ is defined by the relation $d\tau \wedge \be = iB \, dt \wedge \be$ on
$\Ga$ for every $\beta \in K$. We have 
$$ T_{(0,x)} \Ga = \R ( 1, X_x) \oplus \R ( 0, \eta) \oplus \{ ( 0,\xi), \, \xi
\in (T_x \Ga_0) \cap X_x^{\perp_{\om}} \}$$
and $d\tau ( 1, X_x) =0 $ so that $ d\tau ( 0, \eta ) =1 $ and $d \tau ( 0, \xi) =0 $ for any $\xi \in T_x
\Ga_0 \cap X_x^{\perp_{\om}}$ by \eqref{eq:relation_utile}. So evaluating the
relation $ d\tau
\wedge \be = iB \, dt \wedge
\be $ on $(1, X_x) \wedge ( 0, \eta) \wedge (0, \xi_2) \wedge \ldots \wedge
(0,\xi_n)$ where $w = \xi_2 \wedge \ldots \wedge \xi_n$, we get
\begin{gather} \label{eq:unedeplus}
-\be ( X_x \wedge w) =  i B(0,x) \be ( \eta \wedge w ),
\end{gather} 
which gives \eqref{eq:buut}.
The proof for $t \neq 0$ is exactly the same where all the symplectic data $X_x$,
$\eta$, $w$, $\Ga_0$ are replaced by their image under $\phi_t$.

The last point is the determination of the square root: we have $(C'_t(x))^{1/2} =
C_t(x)^{1/2}/ B(t, \phi_t(x))^{1/2}$, with $C_0 (x)^{1/2} =1$ and
the square root $B(t,x)^{1/2}$ has a non negative real part by
Theorem \ref{th:Lagrangian_Fourier}. It is even positive as explained in
Remark \ref{rmk:formula_B}.  
\end{proof}

\begin{rmk}\label{rmk:formula_B} The quantity $C_0'(x) = B(0,x)^{-1}$ can be computed explicitly
  as follows. For $x \in H^{-1} (E) \cap \Ga_0$,
\begin{gather} \label{eq:formula_B}
  B(0,x) = \| X_1 \|^2 + i \om ( X_1, X_2) 
\end{gather}
where $X_x = X_1 + X_2$ with $X_1 \in j_x (T_x \Ga_0)$ and $X_2 \in T_x
  \Ga_0$. Recall that $X_x \notin T_x \Gamma_0$, so $\| X_1 \|^2 \neq 0$.
\end{rmk}
\begin{proof}[Proof of \eqref{eq:formula_B}] We set $\eta = \| X_1\|^{-2}
  j_x X_1$ and compute $B(0,x)$ from \eqref{eq:unedeplus}. On the one hand, $\be$ being
  a $(n,0)$ form, $\be (X_1 \wedge w) = -i \be (jX_1 \wedge w)= - i \|X_1\|^2
  \be ( \eta \wedge w)$. On the other hand, $X_2 = \om ( X_1 , X_2) \eta $ plus a
  linear combination of the $\xi_i$, so $X_2 \wedge w = \om (X_1, X_2) \eta
  \wedge w$. Gathering these equalities we get
  $$ \be (X_x \wedge w ) = \left( \om (X_1, X_2) - i \|X_1 \|^2 \right)   \be ( \eta \wedge
  \om)$$ and the conclusion follows. 
\end{proof}

\subsection{Smoothed spectral projector} \label{sec:smooth-spectr-proj}

Recall that $(T_{k})$ is a self-adjoint
  Toeplitz operator with principal symbol $H$ and subprincipal symbol $H^{\s}$ and that $f\in
  \Ci (\R)$ has a smooth compactly supported Fourier transform.

\begin{thm} \label{th:projector}
  Let  $E$ be a
  regular value of $H$. 
Then the Schwartz
  kernel of $f\bigl( k ( E - T_{k})\bigr)$ is a Lagrangian state
  associated with the Lagrangian immersion $j_E: \Ga^E \rightarrow M^2$, the
  flat section $s^E \in \Ci ( j_E^* (L \boxtimes \overline{L}))$ and the symbol $\si^E \in \Ci (
  j_E^* (L' \boxtimes \overline{L'}))$ given by $\Ga^{E} = \R \times H^{-1} (E)$, $j_E(t,x) =
  (\phi_t(x), x)$, $s^E ( t,x) = \mathcal{T}_t^L (x)$
and $$  \si^E ( t,x) = \hat{f}(t) \bigl[ \rho_t'(x) \bigr]^{\frac{1}{2}} e^{\frac{1}{i} \int_0^t H^{\s} (\phi_r(x)) \;
    dr } \mathcal{T}_t^{L'} (x)
  $$
  where the function $\rho'_t(x)$ is defined below.
\end{thm}
Recall from the introduction the decomposition in symplectic
subspaces $T_xM = F_x \oplus G_x$ where $F_x = \op{Vect} ( X_x, j_x X_x)$ and
$G_x = F_x^{\perp_{\om}}$. $F_x$ and $G_x$ are both preserved by
$j_x$ and we denote by $K(F_x)$, $K(G_x)$ their canonical lines.
We define 
\begin{gather} \label{eq:phiF} 
\Phi_F: K(F_x) \rightarrow K(F_{\phi_t(x)}), \qquad \Phi_F ( \la_x ) = 2
\| X_x \|^{-2} \la_{\phi_t(x)} 
\end{gather}
where $\la_x \in K(F_x)$
is normalised by $\la_x (X_x) =1$. Furthermore $\Phi_G$ is the map $K(G_x)
\rightarrow K(G_{\phi_t(x)})$ such that 
\begin{gather} \label{eq:def_phiG}
  \Phi_G ( \al ) (\psi u ) = \al (u ), \qquad \forall \al \in K(G_x), \, \forall u \in
  \wedge^n T_x M
\end{gather}
where $\psi$ is the symplectic map $G_x \rightarrow G_{\phi_t(x)}$ induced
by $T_x \phi_t$ and  the isomorphism $G_x \simeq T_x H^{-1} (E)/ \R X_x$.

Then we set $  \mathcal{D}'_t(x):= \Phi_F \otimes \Phi_G :K_x \rightarrow
K_{\phi_t(x)} $ and we denote by $\rho'_t (x)$  the complex
number such that $${\mathcal{D}_t' (x)} = \rho'_t(x)  \mathcal{T}^K_{t}
(x).$$
We denote by $\bigl[ \rho'_t(x) \bigr]^{\frac{1}{2}}$ the continuous square root equal
to $\sqrt{2} \| X_x \|^{-1}$ at $t=0$.

\begin{proof} This is a particular case of Proposition
  \ref{prop:lagr-state-spectr} just as Theorem \ref{th:quantum_propagator} on the
  quantum propagator was a particular case of Theorem
  \ref{th:Lagrangian_State_propagation}. Let us compute the coefficient
  $\mathcal{E}'_t(x) ( \op{id}_{K_x} )$. We first describe the image \eqref{eq:def_ltx} of $T_{(t,x)} j_E$:
  $$ \mathfrak{L}_{t,x} = \R (X_x, 0 )  \oplus \{  (T_x \phi_t (\xi), \xi) ,\, \xi \in T_x
  H^{-1} (E) \} $$
  and its volume \eqref{eq:def_vtx}. Set $\eta = \| X_x\|^{-2} j_x X_x$ so
  that $(X_x, \eta)$ is a symplectic basis of $F_x$. Let $(\xi_i)$ be a
  symplectic basis of $G_x$. Then if the volume of $\op{diag} T_xM$ is $v = v_F
  \wedge v_G$ with
  $$v_F = (X_x,
  X_x ) \wedge (\eta, \eta), \qquad v_G =(\xi_1, \xi_1) \wedge \ldots \wedge
  (\xi_{m} , \xi_{m}), $$
  then we have 
$ v(0,x) =  - (X_x, 0 ) \wedge (X_x, X_x )  \wedge (\xi_1, \xi_1) \wedge \ldots \wedge
(\xi_{m} , \xi_{m}) $
so that 
\begin{xalignat*}{2}
v(t,x) &  = - ( X_{\phi_t(x)}, 0 ) \wedge ( X_{\phi_t(x)} , X_x ) \wedge ( T_x \phi_t ( \xi_1),\xi_1)
\wedge \ldots \wedge (T_x \phi_t (\xi_{m}), \xi_{m}) \\
& = - ( X_{\phi_t(x)}, 0 ) \wedge (0,X_x) \wedge ( \psi (\xi_1), \xi_1) \wedge
\ldots \wedge (\psi (\xi_{m}), \xi_m)
\end{xalignat*}
because  $T_x \phi_t (\xi) = \psi (\xi)$ modulo $\R X_{\phi_t(x)}$. 
Then $\mathcal{E}'_t(x) ( \op{id}_{K_x} ) = \Phi_F \otimes \Phi_G$ where
$\Phi_F: K(F_x)  \rightarrow K(F_{\phi_t(x)})$ is
such that 
\begin{gather} \label{eq:FFF}
 \bigl\langle \Phi_F  , (X_{\phi_t(x)}, 0) \wedge (0, X_{x}) \bigr\rangle = i \langle
\op{id}_{K(F_x)}, v_F \rangle, 
\end{gather}
and $\Phi_G: K(G_x)  \rightarrow K(G_{\phi_t(x)})$ is such that 
\begin{gather} \label{eq:GGG}
  \langle \Phi_G , (\psi( \xi_1), \xi_1)
\wedge \ldots \wedge ( \psi (\xi_m), \xi_m) \rangle = \langle
\op{id}_{K(G_x)}, v_G \rangle, 
\end{gather}
Here the pairings are based on the identifications $\op{Mor}
(K(S) , K(S') ) \simeq K(S') \otimes \con{K(S)} \simeq K(S' \oplus \con{S})$.
Now $\Phi_G$ is the application satisfying \eqref{eq:def_phiG} by Lemma
\ref{lem:toutestla}. And $\Phi_F$ is the application \eqref{eq:phiF} by a
straightforward computation. 
\end{proof}

\section{Proof of theorem \ref{th:gamma}} 
\label{sec:proof-theo}

We choose complex normal coordinates $(z_i)$ of $M$ centered at $x_0 \in M$.
So $G_{ij} (x_0) = \delta_{ij}$ and $\partial_{z_i} G_{jk}  (x_0) =
\partial_{\con{z}_i} G_{jk} (x_0) =0$. We may assume that $T_{x_0} \Ga_{t_0} $
is spanned by the vectors $\partial_{z_i} + \partial_{\con{z}_i}$, $i=1, \ldots, n$.
Recall that $Y$ is the vector field $(\partial_t, X_t)$ of $I \times M$ where
$X_t$ is the Hamiltonian vector field of $H_t$. Since
$\om = i \sum_{j,k} G_{jk} dz^j \wedge d \con{z}^k$, we have
\begin{gather} \label{eq:champ_hamiltonien_complex}
X_t = i \sum_{j,k} \bigl( - G^{jk} H_{z_j} \partial_{\con{z}_k} +  G^{jk}
H_{\con{z}_k} \partial_{z_j} \bigr)
\end{gather}
where we use the notation $ H_{z_j}=  \partial_{z_j} H_t$ and
$H_{\con{z}_k} = \partial_{\con{z}_k} H_t$ (and below we will use similar notation for higher order derivatives).
As explained before the statement of Theorem \ref{th:gamma}, we have two derivatives $\nabla_Y $ and $\mathcal{L}_Y$
acting on $(\C_I \boxtimes K)|_{\Ga}$ in the same direction $Y$, so $\theta := \frac{1}{i}
(\mathcal{L}_Y - \nabla_Y)$ is a function in $\Ci ( \Ga)$. 

\begin{prop}\label{prop:delta}
  $\theta (x_0) =  \sum_j\bigl( H_{z_j\con{z}_j} ( x_0)
  +  H_{\con{z}_j\con{z}_j} ( x_0) \bigr) $ 
\end{prop}
\begin{proof}
Let $\al = dz_1 \wedge \ldots \wedge dz_n$. First we have the section $1 \boxtimes \al$
of $(\C_I \boxtimes K)|_{\Ga}$ and we compute its covariant derivative with
respect to $Y$. We claim that this derivative vanishes at $x_0$. This follows from
the fact that $|\al|^{-2} = \det G_{ij}$, so the
Chern connection of $K$ (given near $x_0$ by the one-form $\frac{\partial(|\alpha|^2)}{|\alpha|^2}$) is zero at $x_0$ because the coordinates are normal
at $x_0$.
Second we have to compute the Lie derivative with respect to $Y$
of $j^* (dt \wedge \al)$ with $j$ the embedding $\Ga \rightarrow I \times M$.
We have $\mathcal{L}_Y j^* ( dt \wedge \al) = j^* ( \mathcal{L}_Y ( dt \wedge
\al))$. Furthermore $\mathcal{L}_{\partial_t} dt = \mathcal{L}_{\partial_t}
dz_i = 0$ and by \eqref{eq:champ_hamiltonien_complex}, we have 
$$ \mathcal{L}_{X_t}d z_j = d (X_t. z_j) = i  \sum_k (H_{\con{z}_j z_k } dz_k +
H_{\con{z}_j \con{z}_k } d\con{z}_k) $$
at $x_0$ because the coordinates are normal. Furthermore $j^* (dt \wedge
d\con{z}_k) = j^* ( dt \wedge dz_k)$ at $x_0$ by the assumption on
$T_{x_0} \Ga_{t_0}$. Collecting all these informations, we deduce that 
$$ \mathcal{L}_Y ( j^*(dt \wedge \al) ) = i \sum_j \bigl( H_{z_j\con{z}_j} ( x_0)
+  H_{\con{z}_j\con{z}_j} ( x_0) \bigr)\; j^*(dt \wedge \al) $$
at $x_0$. The conclusion follows.  
\end{proof}

Introduce the Szeg\"o projector $\Pi_k$ which is the orthogonal projector of
$\Ci ( M , L^k \otimes L')$ onto $\Hilb_k$. 
\begin{prop} \label{prop:Lag_state_symbol_calculus}
  For any Lagrangian state $(\Psi_k)$ associated with $\Ga_{t_0}$
  with symbol $\si$ and any function $f \in \Ci ( M)$,
  $\Pi_k (f \Psi_k)$ is a Lagrangian state with symbol
\begin{gather} 
  f|_{\Ga_{t_0}} \, \si + \hbar \bigl( \tfrac{1}{i} \nabla^{L'}_{U} + 
  \square f \bigr) \, \si + \bigo ( \hbar^2) 
\end{gather}
where $U$ is the vector field of $\Ga_{t_0}$ such that $U(x) = X_f (x) \mod
  T_x ^{0,1}M$, $X_f$ being the Hamiltonian vector field of $f$,  and $\square
  f = \sum_j (f_{z_j\con{z}_j} + \tfrac{1}{2} f_{\con{z}_j \con{z}_j})$ at $x_0$. 
\end{prop}

\begin{proof}
  We already know that $\Pi_k (f \Psi_k)$ is a Lagrangian state associated with
  $\Ga_{t_0}$. We compute its symbol up to $\bigo( \hbar^2)$ at $x_0$. It
  suffices to prove that this symbol has the form
\begin{gather} \label{eq:goal}
  (c_0 + \hbar c_1) f(x_0) \si (x_0)  +  \hbar( \tfrac{1}{i} \nabla^{L'}_{U} +
  \square f ) \si (x_0) + \bigo ( \hbar^2)
\end{gather}
where $c_0$ and $c_1$ are independent of $f$. Since for $f=1$, we have to
recover $\si(x_0)$, $c_0=1$ and $c_1 =0$ necessarily.

  Besides
  our normal coordinates $(z_i)$ and our assumption on $T_{x_0} \Ga_{t_0} $,
  let us introduce two holomorphic normal frames $v$ and $v'$ of $L$ and $L'$
  respectively. So $|v|= e^{-\frac{\varphi}{2}} $ with $\varphi $ a real function such
  that
  $$\varphi(x_0) = \partial_{z_j} \varphi(x_0) = \partial_{z_j}
  \partial_{z_k} \varphi (x_0) = 0, \qquad \forall j,k .$$
  Similarly $|v'|= e^{-\frac{\varphi'}{2}} $ with $\varphi'$ satisfying the
  same conditions. Notice that the curvature
of $L$ is equal to both $\partial \con{\partial} \varphi$ and $-i\omega$, so $\partial_{z_i}
\partial_{\con{z}_j} \varphi = G_{ij}$ for every $i,j$.  We can assume that
$v(x_0) = s(x_0)$ where $s$ is the section over $\Ga_{t_0}$ associated with
our Lagrangian state.  
   In the rest of the proof we write all the sections of $L^k \otimes L'$ in the frame $v^k
  \otimes v'$.

More details on the computations to come can be found in \cite[Sections 2.4, 2.5]{C2003}.  We have
  $$
  \Pi_k (x_0, x) = \Bigl( \frac{k}{2\pi} \Bigr)^n e^{  k \psi (x) +  \psi'(x) }
  p(x,k) + \bigo( k^{-\infty}) $$
where $\psi$ has the following Taylor expansion at $x_0$: $\psi (x) = \sum_{\be \in
  \N^n}\varphi_{0,\be}
\frac{\con{z}^\be}{\be!}$ with the notations  $\varphi_{\al, \be} =
\partial^\al_z \partial_{\con{z}}^\be \varphi (x_0)$, $\psi'$ has
the same Taylor expansion in terms of $\varphi'$ and $p(x,k) =  1 + k^{-1}
p_1 ( x)  + k^{-2} p_2(x) + \ldots$.

We have a similar expression for $\Psi_k$:
$$ \Psi_k ( x) = \Bigl( \frac{k}{2\pi} \Bigr)^{\frac{n}{4}} e^{k \rho (x) } a(x,k) + \bigo( k^{-\infty}) $$
where $\rho$ has the Taylor expansion $\rho (x) =  \frac{1}{2} \sum z_i^2 +
\sum_{|\al| \geqslant 3} \rho_{\al, 0} \frac{z^\al}{\al!}$. This follows from
the fact that the section $F$ entering in the definition
\eqref{eq:def_lagrangian} of $\Psi_k$ satisfies $F(x_0) = s(x_0)$ so that we
can assume that $\rho (x_0) =0$, $\nabla F |_{x_0} =0$ so that the first
derivatives of $\rho$ all vanish at $x_0$ and finally the second order
derivatives of $F$ at $x_0$ depend only on the linear data at $x_0$ \cite[Proposition
2.2]{C2003} which leads to the expression above. 

So it comes that
\begin{gather} \label{eq:uneintdeplus}
\Pi_k ( f \Psi_k) ( x_0) = \Bigl( \frac{k}{2\pi} \Bigr)^{\frac{5n}{4}}  \int e^{-k
  \phi } f(x) a(x,k) p(x,k) D(x) \, dz d\con{z}   + \bigo( k^{-\infty})
\end{gather} 
where $$\phi (x) = - \psi (x) + \varphi (x) - \rho (x) = |z|^2 - \tfrac{1}{2}
\sum z_i^2 + R (x)$$ 
with $R$ having the Taylor expansion at $x_0$
$$ R(x) = \sum_{\substack{\al\neq 0, \be \neq 0\\|\al| + |\be| \geqslant 3 }} \varphi_{\al, \be}
\frac{z^\al \con {z}^\be}{\al! \be !}   + \sum_{|\al| \geqslant 3}( -\rho_{\al,
  0}  + \varphi_{\al, 0 } )\frac{z^\al}{ \al!}.$$
Furthermore $ D(x) = e^{\psi'(x) -\varphi'(x) }\det (G_{ij}) = 1 + \bigo
(|z|^2).$

By applying the stationary phase method, we obtain the asymptotic expansion of
\eqref{eq:uneintdeplus}. At first order: 
$$ \Pi_k ( f \Psi_k) ( x_0) = C \Bigl( \frac{k}{2\pi}
\Bigr)^{\frac{n}{4}}  \left( f(x_0) a(x_0, k) + \bigo (k^{-1}) \right)$$
  where $C$ can be computed in terms of the Hessian determinant of
  $\phi$ at $x_0$. Actually it is shorter to compare with \eqref{eq:goal},
  which gives  $C=c_0 =1$.
  We now compute the second order term,
  $$ \Pi_k ( f \Psi_k) ( x_0) =  \Bigl( \frac{k}{2\pi}
\Bigr)^{\frac{n}{4}}  \left( f(x_0) a(x_0, k) + k^{-1} \varepsilon_k + \bigo (
k^{-2} )\right) $$ with
\[ \varepsilon_k =  \sum_{\ell=1}^{3} (\ell! \, (\ell-1) !)^{-1} P^{\ell} ( (-R)^{\ell -1} f a(\cdot ,k)
D ) (x_0)  \]
where $P = \sum ( \partial_{z_j} \partial_{\con{z}_j} + \tfrac{1}{2}
\partial_{\con{z}_j} \partial_{\con{z}_j})$.

Recall that we do not try to compute the terms $c_0$ and $c_1$ in
\eqref{eq:goal}, which means that we do not take into account the terms without
derivative on $a(x,k)$. Considering the form of the Taylor expansions of $R$
and $D$ given above and the fact that $\partial_{z_i}
\partial_{\con{z}_j} \varphi =
G_{ij}$ so that $\varphi_{\al, \be} =0$ when $(|\al|, |\be|) = (1,2)$ or
$(2,1)$,  we deduce after some investigations that we only have a single term
to consider which is $k^{-1} P ( f(x) a(x,k))$. Now $a(x,k) = a_0 (x) + \bigo
( k^{-1})$ and $a_0$ has the Taylor expansion of a holomorphic function at
$x_0$ because $\con \partial ( a_0 v')$ vanishes to infinite order along
$\Ga_{t_0}$ and the frame $v'$ is holomorphic. So
\begin{gather} \label{eq:toutcapourca}
P( f a_0) = (Pf) a_0+ \sum_j (\partial_{\con{z}_j} f) (\partial_{z_j}
a_0). 
\end{gather}
The Hamiltonian vector field of $f$ at $x_0$ being $\sum_j (- i f_{z_j}
\partial_{\con{z}_j} + i f_{\con{z}_j} \partial_{z_j})$ (see
\eqref{eq:champ_hamiltonien_complex}), and by using again that $\con \partial
a_0 =0$ at $x_0$, we obtain that the sum in \eqref{eq:toutcapourca} is
$\frac{1}{i} X_f a_0 = \frac{1}{i} U a_0$. Now the frame $v'$ being normal, its covariant derivative is
$0$ at $x_0$, so $\nabla_U ( a_0 v') = (U. a_0) v'$. 
\end{proof}

\begin{proof}[Proof of Theorem \ref{th:gamma}]
It suffices to consider the case where $T_{k,t} = T_k(H_t + k^{-1} H_t')$ for
some $H_t, H_t' \in \mathcal{C}^{\infty}(M)$. But the symbol of $k^{-1} T_k
(H_t') \Psi_k$ is $\hbar H_t' \sigma + \bigo(\hbar^2)$, so in fact we only
need to consider the case $H_t' = 0$. 

By Proposition
\ref{prop:Lag_state_symbol_calculus} and  the proof of Proposition \ref{prop:derive_etat_lag}, the symbol of $\frac{1}{ik} \partial_t
\Psi_k + T_k (H_t)  \Psi_k$ is
\begin{gather} \label{eq:hbar-tfrac1i-nabla_z}
\hbar \bigl( \tfrac{1}{i} (\nabla_Z +  \nabla_{U} ) +
  \square H_t  \bigr) \, \sigma +
  \bigo( \hbar^2), 
\end{gather}
where $Z$ and $U$ are the vector fields of $\Ga$ such that
  $Z(t,x) = \partial_t \mod T^{0,1} _xM$ and $U(t,x) = X_t(x) \mod
  T^{0,1}_xM$. Since here $Y(t,x) = (\partial_t, X_t)$ is tangent to $\Ga$, we
  have $ Z +  U =Y$ on $\Ga$.  By Proposition \ref{prop:delta}, $\square
H_t = \frac{1}{2} \theta + \frac{1}{2} \Delta H_t$ with $\Delta = \sum_{i,j} G^{ij}
\partial_{z_i} \partial_{\con{z}_j}$ (indeed, recall that the coordinates are normal at $x_0$ so $\Delta H_t(x_0) = \sum_{i}
H_{z_i \con{z}_i } (x_0)$). So the symbol in
\eqref{eq:hbar-tfrac1i-nabla_z} is
$$  \hbar \bigl( \tfrac{1}{i} \nabla_Y + \tfrac{1}{2} \theta + \tfrac{1}{2} \Delta
H_t \bigr) \, \si + \bigo ( \hbar^2)$$
which concludes the proof. 
\end{proof}

\appendix 

\section{Appendix: an explicit example}

Let $(M,\omega) = (\T^2 = \R^2 \slash \Lambda, \omega_{\T^2})$ where $\Lambda \subset \R^2$ is a lattice of symplectic volume $4\pi$. $(M,\omega)$ is naturally endowed with a prequantum line bundle $L$ induced by the line bundle $\R^2 \times \C \to \R^2$ with connection $\nabla = d - i \alpha$ where $\alpha = 2 \pi (p \ dq - q \ dp)$. Here $(p,q)$ are coordinates associated with a basis $(e,f)$ of $\Lambda$ with $\omega(e,f) = 4\pi$. In other words $\omega = 4\pi dp \wedge dq$; we will also work with the holomorphic coordinate $z = p + i q$, for which $\omega = 2 i \pi dz \wedge d\bar{z}$.

For $k \geq 1$, the quantum space $\Hil_k = H^{0}(M,L^k)$ identifies with the space of $\Lambda$-invariant sections of $\R^2 \times \C \to \R^2$, which is a space of theta functions with dimension $2k$. More precisely, the family $(\Psi_{\ell})_{0 \leq \ell \leq 2k-1}$ given by
\[ \Psi_{\ell}(z) = \frac{k^{\frac{1}{4}}}{\sqrt{2\pi}} \exp(2i\pi(\ell + k \Im(z))) \exp\left(-\frac{\pi \ell^2}{2k}\right) \vartheta_3(\pi(2kz + i \ell), \exp(-2k\pi)) \]
for all $z \in \C$ and $\ell \in \{0, \ldots, 2k-1 \}$, where $\vartheta_3(w,q) = 1 + 2 \sum_{n=1}^{+\infty} q^{n^2} \cos(2nw)$ is the Jacobi theta function, forms an orthonormal basis of $\Hil_k$. The diagonal operator defined as $T_k \Psi_{\ell} = \cos(\frac{\pi \ell}{k}) \Psi_{\ell}$ for every $\ell \in \{0, \ldots, 2k-1 \}$ is a Berezin-Toeplitz operator with principal symbol $H: (p,q) \mapsto \cos(2\pi q)$ and vanishing subprincipal symbol. For more details, see \cite[Section 3.1, Appendix]{CM}.

On the one hand, one can easily compute numerically the kernel of the quantum propagator $U_{k,t} = \exp(-i k t T_k)$ by using the formula
\begin{equation} U_{k,t}(w,z) = \sum_{\ell = 0}^{2k-1} \exp\left(-ikt \cos\left(\frac{\pi \ell}{k}\right)\right) \Psi_{\ell}(w) \overline{\Psi_{\ell}(z)}  \label{eq:kernel_torus} \end{equation}
and the above expression of $\Psi_{\ell}$ (in practice, we use the built-in commands for Jacobi theta functions in the mpmath library for Python). On the other hand, the coefficients in Equation \eqref{eq:leading_order_prop_quant} can be explicitly computed as follows. Since the subprincipal symbol of $T_k$ vanishes, it suffices to compute $\rho_t$ and $\phi_t^L$. First, the parallel transport term reads
\[ \mathcal{T}_t^L(p,q) = \exp\left(i \int_0^t \alpha_{\phi_s(p,q)}(X(\phi_s(p,q))) ds \right) = \exp\left(- i \pi t q \sin(2 \pi q) \right)  \]
since $X(p,q) = \frac{1}{2} \sin(2\pi q) \partial_p$ and $\phi_t(p,q) = [p + \frac{t}{2} \sin(2\pi q), q]$, and we obtain 
\[ \phi_t^L(p,q) = \exp(-it\cos(2\pi q)) \mathcal{T}_t^L(p,q) = \exp\left(- i t \left(\cos(2\pi q) + \pi q \sin(2 \pi q) \right) \right). \]
Second, one readily checks that $(T_{(p,q)} \phi_t)^{1,0}$ is the operator of multiplication by $1 - \frac{i\pi t}{2} \cos(2 \pi q)$. Since the connection on the canonical bundle is trivial, this yields
\[ \rho_t(p,q)^{\frac{1}{2}} = \left( 1 - \frac{i\pi t}{2} \cos(2 \pi q) \right)^{-\frac{1}{2}} =  \frac{\exp\left( \frac{i}{2} \arctan\left( \frac{\pi t}{2} \cos(2\pi q) \right) \right)}{\sqrt{1 + \frac{\pi^2 t^2}{4} \cos^2(2\pi q) }}.\]
So we finally obtain that for $z = p + i q$,
\begin{equation} U_{k,t}(\phi_t(z),z) \sim \frac{k \exp\left( i \left( \frac{1}{2} \arctan\left( \frac{\pi t}{2} \cos(2\pi q) \right) - kt \left( \cos(2\pi q) + \pi q \sin(2\pi q) \right) \right) \right)}{2\pi \sqrt{1 + \frac{\pi^2 t^2}{4} \cos^2(2\pi q) }} .  \label{eq:kernel_torus_theo}  \end{equation}

We compare this theoretical equivalent with the numerical value in Figures \ref{fig:prop_real_0301}, \ref{fig:prop_imag_0301}, \ref{fig:prop_real_0507}, \ref{fig:prop_real_0301_small_times} and \ref{fig:prop_real_0301_large_times}. In these computations, we fix $k$ and $(p,q)$, and plot the real part of the kernel of the propagator evaluated at $(\phi_t(z),z)$ with $z = p+iq$, as a function of $t$; we also plot the imaginary part of this kernel only for one set of parameters, since the behaviour is very similar to the one of the real part. In all these figures, the blue diamonds represent the numerical values obtained from Equation \eqref{eq:kernel_torus} while the solid red line corresponds to the right hand side of Equation \eqref{eq:kernel_torus_theo}. Note that a priori the $\bigo(k^{-1})$ remainder may depend on $t$, so once $k$ and $(p,q)$ are fixed, the approximation may become less precise as $t$ increases. In Figures \ref{fig:prop_real_0301_small_times} and \ref{fig:prop_real_0301_large_times}, we display the behaviour at small and (relatively) large times. Investigating the $k$-dependent times up to which the approximation in Equation \eqref{eq:leading_order_prop_quant} remains valid is a classical topic in the semiclassical literature, that we do not consider in the present paper.

\begin{center}
\begin{figure}[H]
\includegraphics[scale=0.35]{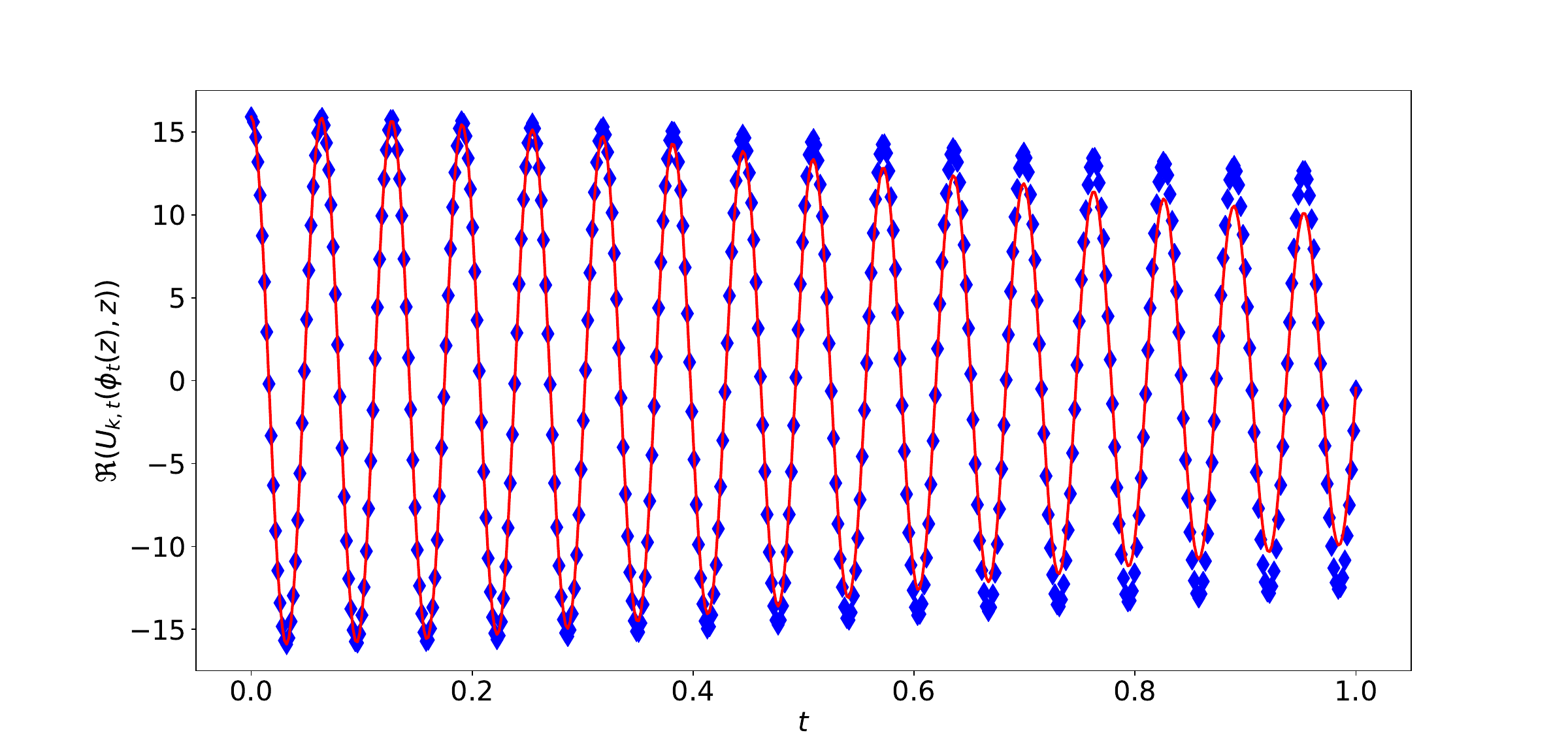}
\caption{Real part of $U_{k,t}(\phi_t(z),z)$ for $k = 100$ and $z=p+iq$ with $(p,q) = (0.3,0.1)$, for $0 \leq t \leq 1$.}
\label{fig:prop_real_0301}
\end{figure}
\end{center}

\begin{center}
\begin{figure}[H]
\includegraphics[scale=0.35]{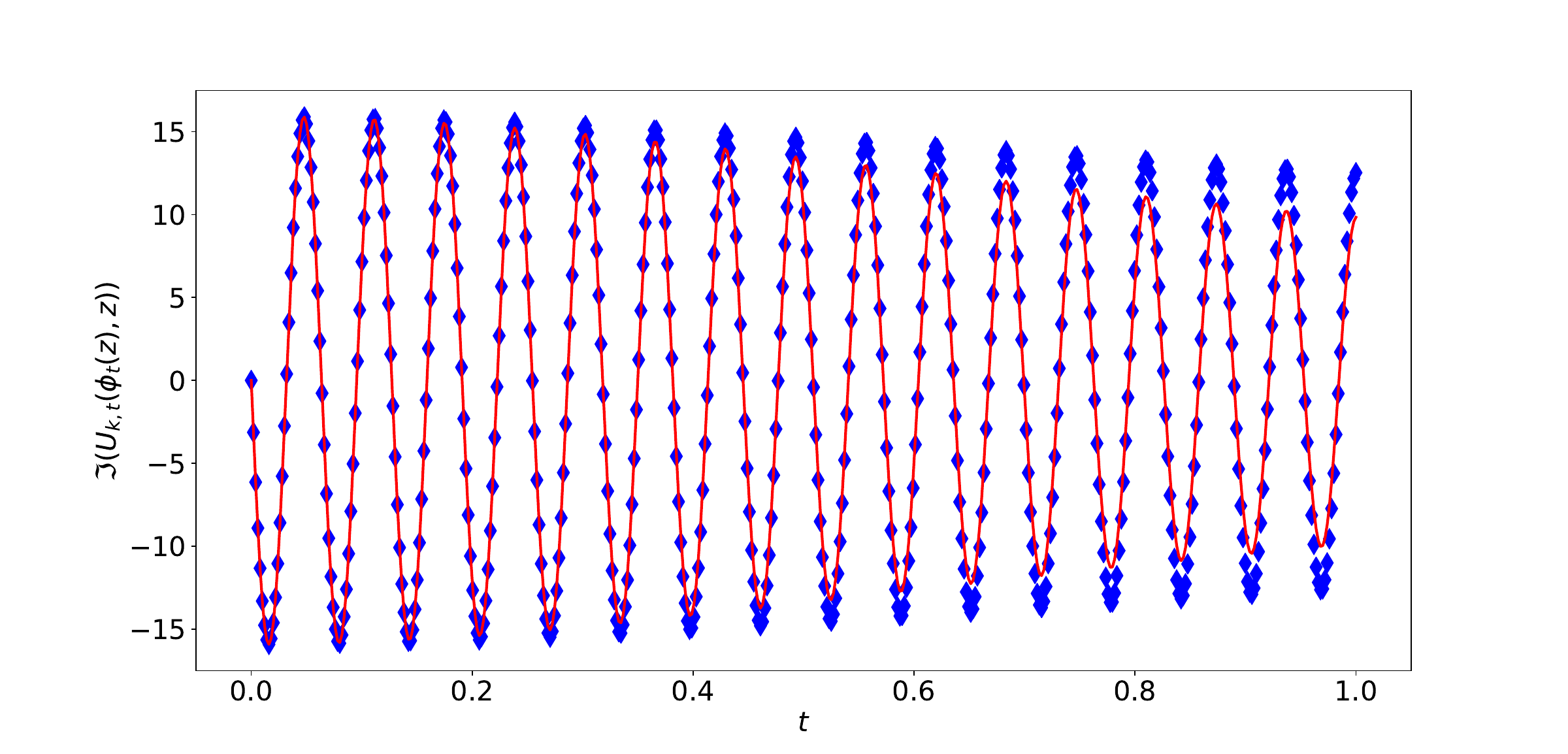}
\caption{Imaginary part of $U_{k,t}(\phi_t(z),z)$ for $k = 100$ and $z=p+iq$ with $(p,q) = (0.3,0.1)$, for $0 \leq t \leq 1$.}
\label{fig:prop_imag_0301}
\end{figure}
\end{center}

\begin{center}
\begin{figure}[H]
\includegraphics[scale=0.35]{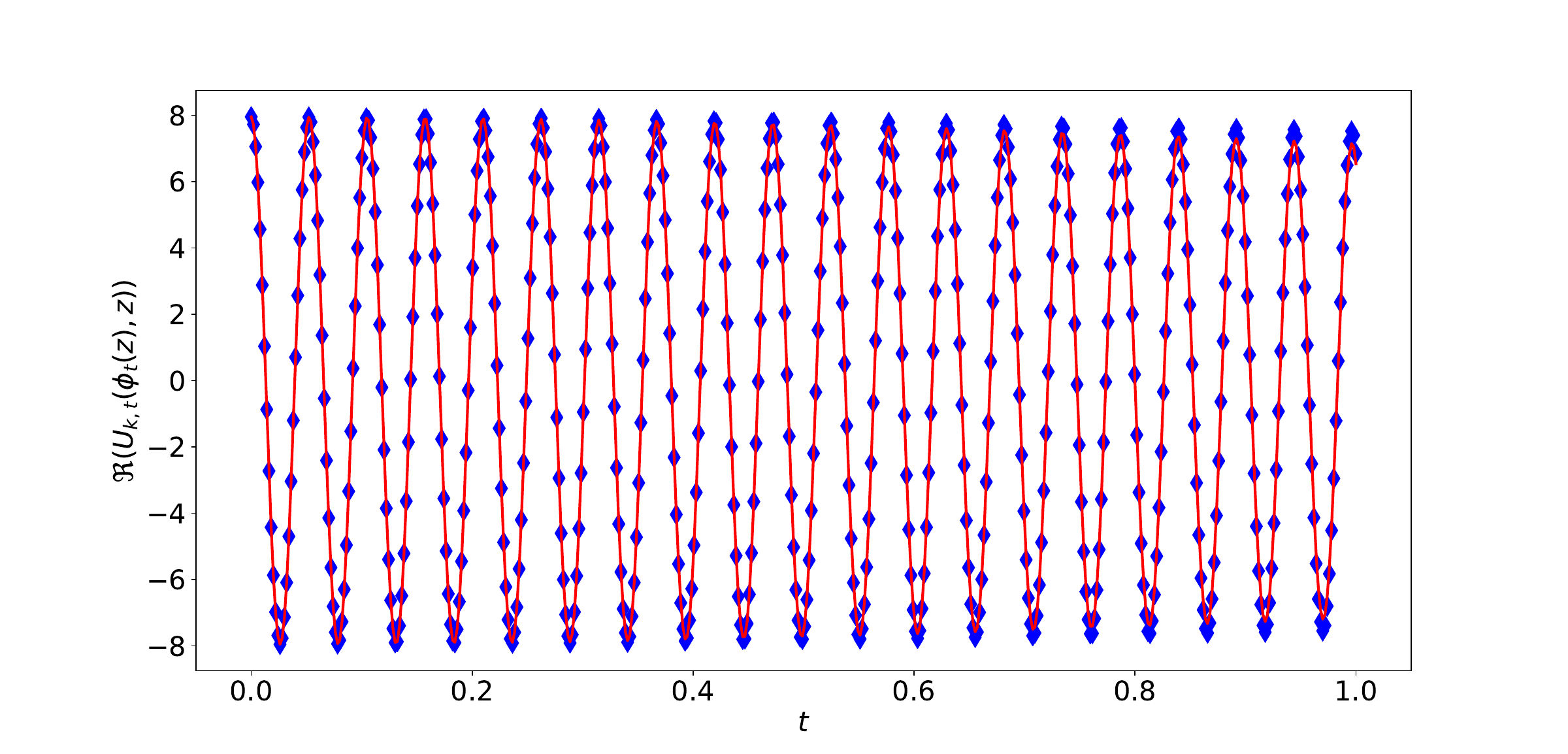}
\caption{Real part of $U_{k,t}(\phi_t(z),z)$ for $k = 50$ and $z=p+iq$ with $(p,q) = (0.5,0.7)$, for $0 \leq t \leq 1$.} 
\label{fig:prop_real_0507}
\end{figure}
\end{center}

\begin{center}
\begin{figure}[H]
\includegraphics[scale=0.3]{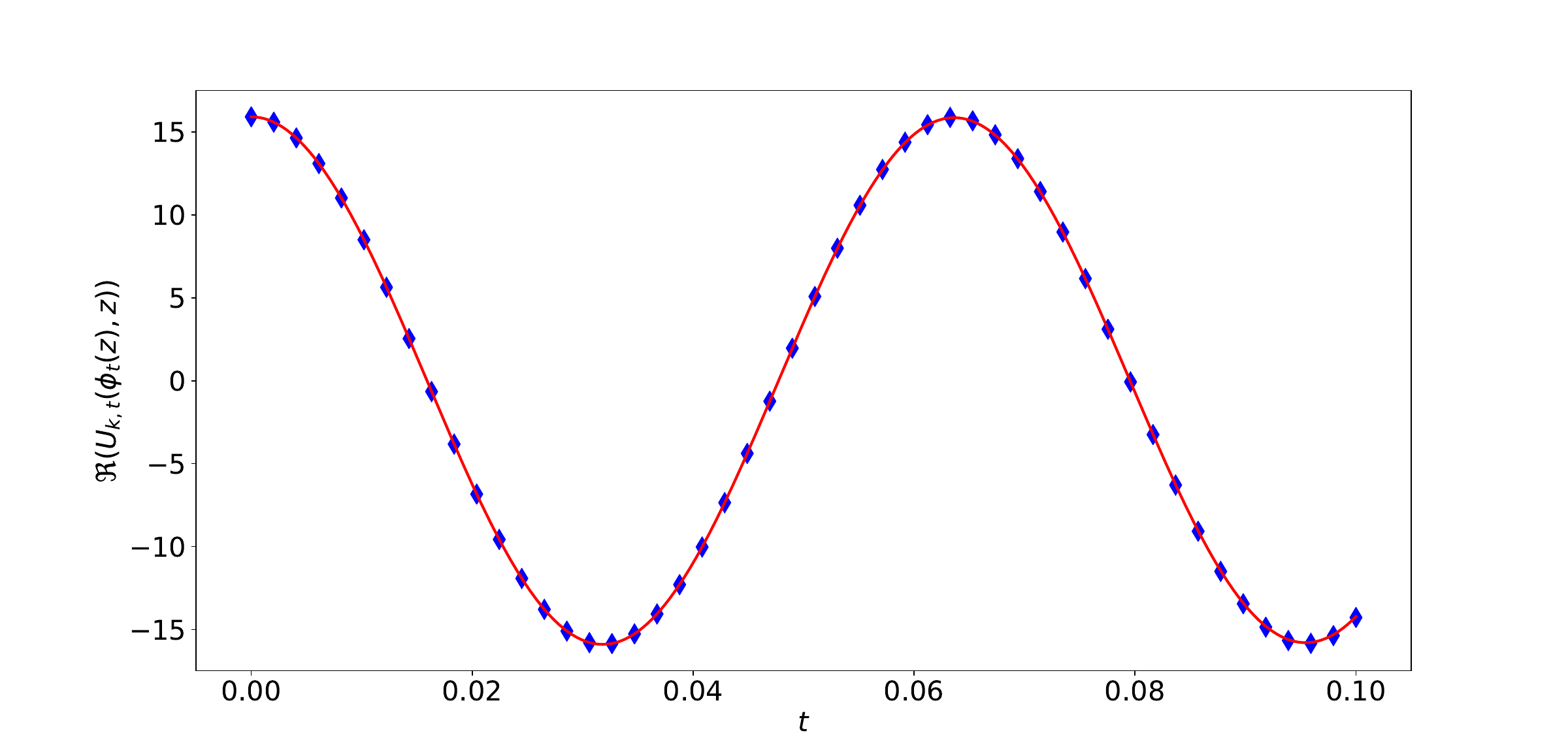}
\caption{Real part of $U_{k,t}(\phi_t(z),z)$ for $k = 100$ and $z=p+iq$ with $(p,q) = (0.3,0.1)$, for $0 \leq t \leq 0.1$.}
\label{fig:prop_real_0301_small_times}
\end{figure}
\end{center}

\begin{center}
\begin{figure}[H]
\includegraphics[scale=0.3]{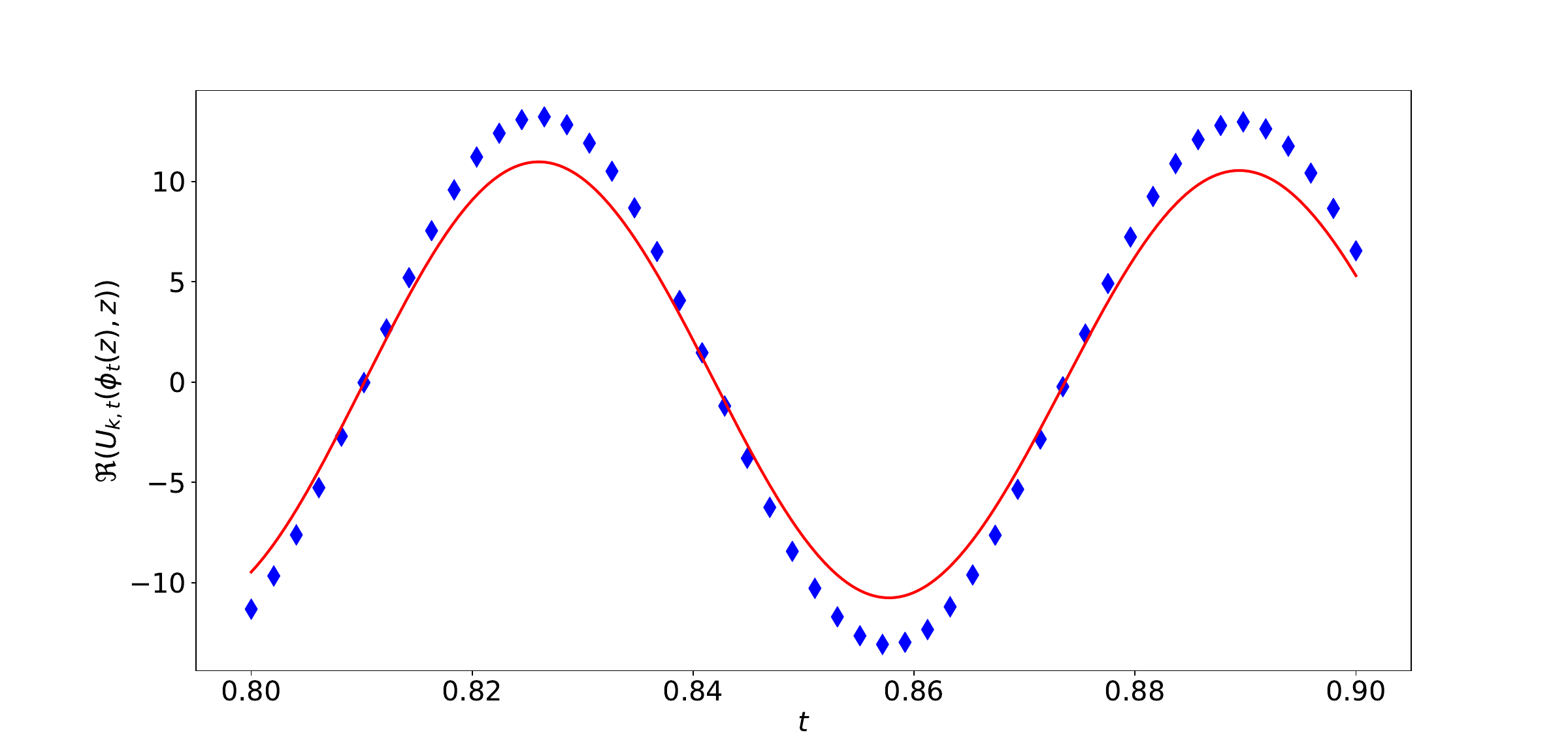}
\caption{Real part of $U_{k,t}(\phi_t(z),z)$ for $k = 100$ and $z=p+iq$ with $(p,q) = (0.3,0.1)$, for $0.8 \leq t \leq 0.9$.}
\label{fig:prop_real_0301_large_times}
\end{figure}
\end{center}

 \bibliographystyle{abbrv}
\bibliography{propagation}

\vspace{15mm}

\noindent
{\bf Laurent Charles} \\
Sorbonne Université, Université de Paris, CNRS \\
Institut de Mathématiques de Jussieu-Paris Rive Gauche \\
F-75005 Paris, France \\
{\em E-mail:} \texttt{laurent.charles@imj-prg.fr}\\

\noindent
{\bf Yohann Le Floch} \\
Institut de Recherche Math\'ematique Avanc\'ee\\
UMR 7501, Universit\'e de Strasbourg et CNRS\\
7 rue Ren\'e Descartes\\
67000 Strasbourg, France\\
{\em E-mail:} \texttt{ylefloch@unistra.fr}

\end{document}